\documentclass[reqno]{amsart}
\usepackage{amsmath,amssymb,epsfig}
\usepackage[pdftex,pdftitle={Doubly nonlinear chemotaxis model},colorlinks=true,breaklinks=true,linkcolor=blue]{hyperref}
\newcommand{\norm}[1]{\ensuremath{\left\|#1\right\|}}
\newcommand{\abs}[1]{\ensuremath{\left|#1\right|}}
\newcommand{\Om}{\ensuremath{\Omega}}

\newcommand{\K}{\ensuremath{\mathcal{K}}}
\newcommand{\A}{\ensuremath{\mathcal{A}}}
\newcommand{\D}{\ensuremath{\mathcal{D}}}

\newcommand{\eps}{\ensuremath{\varepsilon}}
\newcommand{\R}{\ensuremath{\mathbb{R}}}
\newcommand{\N}{\ensuremath{\mathbb{N}}}

\newcommand{\meas}{\mathrm{meas}}
\newcommand{\Div}{\mathrm{div}\,}
\newcommand{\Grad}{\mathrm{\nabla}}

\newcommand{\uu}{\ensuremath{\bar{u}}}
\newcommand{\pt}{\ensuremath{\partial_t}}
\newcommand{\dx}{\ensuremath{\, dx}}

\newcommand{\dt}{\ensuremath{\, dt}}
\newcommand{\ds}{\ensuremath{\, ds}}

\newcommand{\esssup}{\operatorname*{ess\,sup}}
\newcommand{\essinf}{\operatorname*{ess\,inf}}
\newcommand{\essosc}{\operatorname*{ess\,osc}}
\newcommand{\pdist}{\operatorname*{{\it p}-dist}}

\newcommand{\EE}{\mathcal{E}}
\newcommand{\TT}{\mathcal{T}}
\allowdisplaybreaks
\newcounter{dctr}[section]
\numberwithin{equation}{section}

\newtheorem{corollary}{Corollary}[section]
\newtheorem{theorem}{Theorem}[section]
\newtheorem{definition}{Definition}[section]
\newtheorem{lemma}{Lemma}[section]
\newtheorem{proposition}{Proposition}[section]
\newtheorem{rem}{Remark}[section]

\title[Doubly nonlinear chemotaxis model]
{On a doubly nonlinear diffusion model of chemotaxis\\ with
prevention of overcrowding}
\author[Bendahmane]{Mostafa Bendahmane$^{\mathrm{a}}$}
\author[B\"urger]{Raimund B\"urger$^{\mathrm{a}}$}
\author[Ruiz]{Ricardo Ruiz Baier$^{\mathrm{a}}$}
\author[Urbano]{Jos\'e Miguel Urbano$^{\mathrm{b}}$}
\thanks{$^{\mathrm{a}}$Departamento de Ingenier\'{\i}a Matem\'{a}tica,
 Universidad de Concepci\'{o}n, Casilla 160-C,  Concepci\'{o}n,
 Chile. \newline
E-mail: {\tt mostafab@ing-mat.udec.cl}, {\tt rburger@ing-mat.udec.cl},
{\tt rruiz@ing-mat.udec.cl}}
\thanks{$^{\mathrm{b}}$  CMUC, Department of Mathematics, University of
Coimbra, 3001-454 Coimbra, Portugal.
E-mail: {\tt jmurb@mat.uc.pt}}
\date{\today}
\begin{document}
\begin{abstract}
This paper addresses the existence and regularity of weak solutions
for a fully parabolic model of chemotaxis, with prevention of
overcrowding, that degenerates in a two-sided fashion, including an
extra nonlinearity represented by a $p$-Laplacian diffusion term. To
prove the existence of weak solutions, a Schauder fixed-point
argument is applied to a regularized problem and the compactness
method is used to pass to the limit. The local H\"older regularity
of weak solutions is established using the method of intrinsic
scaling. The results are a contribution to showing, qualitatively,
to what extent the properties of the classical Keller-Segel
chemotaxis models are preserved in a more general setting. Some
numerical examples illustrate the model.
\end{abstract}
\maketitle
\section{Introduction}
\subsection{Scope}
It is the purpose of this paper to study the  existence and regularity
 of weak solutions of the following parabolic system,
 which is a generalization of the  well-known
 Keller-Segel   model \cite{Horstmann:03,Keller-Segel:71,murray03}
 of chemotaxis:
\begin{subequations}
\label{bbru_S1-S2-S3}
\begin{align}
&\pt u-
\Div\bigl(|\Grad A(u)|^{p-2}\Grad A(u)\bigr)
+\Div\bigl(\chi uf(u)\Grad v\bigr )=0 \quad \text{in $Q_{T}:=\Om
  \times (0,T)$}, \quad
  T>0, \quad \Om\subset\mathbb{R}^N,   \label{bbru_eq1.1a}  \\
& \pt v - d \Delta v =g(u,v) \quad\text{in
  $Q_{T}$},   \label{bbru_eq1.1b} \\
& |\Grad A(u)|^{p-2}a(u)\frac{\partial u}{\partial \eta}=0,\quad
\frac{\partial v}{\partial \eta}=0 \quad \text{on $\Sigma_T:=\partial \Om \times (0,T)$} , \\
& u(x,0)=u_0(x),\quad v(x,0)=v_0(x) \quad \text{on $\Omega$,}
\end{align}
  \end{subequations}
where $\Om\subset\mathbb{R}^N$ is a bounded domain  with a
sufficiently smooth boundary $\partial \Om$ and outer unit
normal~$\eta$.
 Equation~\eqref{bbru_eq1.1a} is {\em doubly nonlinear}, since we
 apply the $p$-Laplacian  diffusion operator, where we assume
 $2\le p<\infty$, to the integrated diffusion function
 $A(u) := \int_0^u a(s) \, ds$, where $a(\cdot)$ is a non-negative
 integrable function with support on the interval~$[0,1]$.

In the biological
phenomenon described by   \eqref{bbru_S1-S2-S3}, the
 quantity $u=u(x,t)$ is the density of
 organisms, such as  bacteria or
cells. The  conservation PDE \eqref{bbru_eq1.1a}
incorporates two competing mechanisms, namely the density-dependent
  diffusive motion  of the cells, described by the doubly
 nonlinear diffusion term, and a
 motion  in response to and towards the  gradient $\nabla v$ of the
 concentration $v=v(x,t)$ of
a substance called {\em chemoattractant}.
  The movement in response to $\nabla v$  also involves the
 density-dependent probability  $f(u(x,t))$
 for
 a cell located at $(x,t)$ to find space in a  neighboring location,
 and a constant $\chi$  describing   chemotactic sensitivity.
 On the other hand, the PDE  \eqref{bbru_eq1.1b}  describes the
 diffusion of the  chemoattractant, where  $d>0$ is
 a diffusion constant
and the function $g(u,v)$ describes the rates of production and
degradation of the chemoattractant; we here adopt 
 the  common choice
\begin{align} \label{bbru_gchoice}
g(u,v)=\alpha u - \beta v, \quad \alpha,\beta\ge 0.
\end{align}

We assume that there exists a maximal population density of
cells~$u_{\mathrm{m}}$ such that $f(u_{\mathrm{m}})=0$. This
corresponds to a switch to repulsion at high densities, known as
prevention of overcrowding, volume-filling effect or density control
(see \cite{bku}). It means that cells stop to accumulate at a given
point of $\Om$ after their density attains a certain threshold
value,  and the chemotactic cross-diffusion  term  $\chi u f(u)$
vanishes identically when $u \ge u_{\mathrm{m}}$. We also assume
that the diffusion coefficient $a(u)$ vanishes at $0$ and
$u_{\mathrm{m}}$, so that \eqref{bbru_eq1.1a} degenerates for $u=0$
and $u=u_{\mathrm{m}}$, while $a(u)>0$ for $0<u<u_{\mathrm{m}}$. A
typical example is  $a(u)=\epsilon u(1-u_{\mathrm{m}})$,
 $\epsilon >0$.   Normalizing variables by
$\tilde{u}=u/u_{\mathrm{m}}$, $\tilde{v}=v$ and
$\tilde{f}(\tilde{u})=f(\tilde{u}u_{\mathrm{m}})$, we have
$\tilde{u}_{\mathrm{m}}=1$; in the sequel we will omit tildes in the
notation.

The main intention  of the present work is  to  address the
question of the regularity of weak solutions, which is a delicate
analytical issue since  the structure of equation  \eqref{bbru_eq1.1a}
 combines a
degeneracy of $p$-Laplacian type with a two-sided point degeneracy
in the diffusive term. We prove the local H\"older continuity of the
weak solutions of \eqref{bbru_S1-S2-S3} using the method of
intrinsic scaling (see \cite{DiBe,Urb-book}). The novelty lies in
tackling the two types of degeneracy simultaneously and finding the
right geometric setting for the concrete structure of the PDE.
 The resulting  analysis
 combines the technique used by Urbano \cite{Urb} to
study the case of a diffusion coefficient $a(u)$
 that decays like a power at both
degeneracy points (with $p=2$) with    the technique by Porzio and
Vespri \cite{Porz-Vesp} to study  the $p$-Laplacian, with
$a(u)$ degenerating at only one side. We recover both results as
particular cases of the one studied here. To our knowledge, the
$p$-Laplacian is a new ingredient in chemotaxis models,
 so we also include
a few numerical examples that illustrate the behavior of solutions of
\eqref{bbru_S1-S2-S3} for $p>2$, compared with solutions to the
standard case $p=2$, but including nonlinear diffusion.

\subsection{Related work}
To put this paper in the proper perspective,
 we recall  that the Keller-Segel model is a widely
studied topic, see  e.g. Murray \cite{murray03} for a general
 background and Horstmann
 \cite{Horstmann:03}  for a fairly complete
survey on the Keller-Segel model and  the variants that have been
 proposed. Nonlinear diffusion
 equations for biological populations that degenerate at least for
 $u=0$   were proposed in the 1970s by   Gurney and Nisbet
 \cite{gurney} and  Gurtin and McCamy \cite{gurtin}; more recent
 works include those
by Witelski \cite{witelski}, Dkhil \cite{dkhil}, Burger et al.\
\cite{Burger-al} and Bendahmane et al.\ \cite{bku}.  Furthermore,
 well-posedness  results
 for  these kinds of models include, for example, the existence of radial
solutions exhibiting chemotactic collapse \cite{Her-Vel},
the local-in-time existence, uniqueness and positivity of classical solutions,
and  results on their blow-up behavior \cite{Yagi},
 and  existence and uniqueness using the abstract theory developed in
 \cite{Amann}, see \cite{Laurencot-Wrzosek}. Burger et al.\  \cite{Burger-al}
  prove  the global existence
and uniqueness of the Cauchy problem in $\R^N$ for
linear and nonlinear diffusion
with prevention of overcrowding.  The model proposed  herein exhibits
an even higher degree of nonlinearity, and offers further
possibilities to describe chemotactic movement; for example,
one could imagine that the cells or bacteria are actually placed
 in a medium with a non-Newtonian rheology.  In fact, the
 evolution $p$-Laplacian equation
 $u_t = \Div (| \nabla u |^{p-2} \nabla u)$, $p >1$, is also called
 non-Newtonian filtration equation, see \cite{drabek} and
 \cite[Chapter~2]{FourChineseBook} for surveys.
 Coming back to the  Keller-Segel model,
 we also  mention 
 that another effort to endow
 this model with a more general diffusion mechanism
 has recently been made by Biler and Wu \cite{bilerwu},
 who consider fractional diffusion.

Various results on  the H\"older regularity of weak solutions to
quasilinear parabolic systems are based  on the work of  DiBenedetto
\cite{DiBe}; the present  article also contributes to this
direction. Specifically for a chemotaxis model, Bendahmane, Karlsen,
and Urbano   \cite{bku} proved the existence and H\"older regularity
of weak solutions for a version of  \eqref{bbru_S1-S2-S3} for $p=2$.
For a detailed description of the intrinsic scaling method and some
applications we refer to the books \cite{DiBe,Urb-book}.

Concerning uniqueness of solution, the presence of a nonlinear
degenerate diffusion term and a nonlinear transport term represents
a disadvantage and we could not obtain the uniqueness of a weak
solution. This contrasts with the results by Burger et al.\  
\cite{Burger-al}, where the authors   prove
uniqueness of solutions for a degenerate parabolic-elliptic system
set in an unbounded domain, using a method which relies on a
continuous dependence estimate from  \cite{KR:Rough_Unique}, that
does not apply to our problem because it is difficult to bound
$\Delta v$ in $L^\infty(Q_T)$ due to the parabolic nature of 
\eqref{bbru_eq1.1b}.
\subsection{Weak solutions and statement of main results}
Before stating our main results, we give  the definition of a weak
solution to \eqref{bbru_S1-S2-S3},  and recall the notion of certain
functional spaces.
 We denote by  $p'$ the conjugate exponent of $p$ (we will restrict ourselves
to the degenerate case $p\ge 2$):
$\frac{1}{p}+\frac{1}{p'}=1$. Moreover,
 $C_w(0,T,L^2(\Om))$ denotes the space of continuous
functions with values in (a closed ball of) $L^2(\Om)$ endowed
with the weak topology, and $\left \langle \cdot,\cdot \right
\rangle$  is  the  duality pairing between $W^{1,p}(\Om)$
and its dual  $(W^{1,p}(\Om))'$.

\begin{definition}
\label{bbru_def1}
A weak solution of \eqref{bbru_S1-S2-S3} is a pair $(u,v)$ of
functions satisfying the following conditions: {\em
\begin{align*}
    & \text{$0\le u(x,t)\le 1$ and $v(x,t)\ge 0$ for a.e.~$(x,t)\in Q_T$},\\
    & u\in C_w\bigl(0,T,L^2(\Om)\bigr), \quad \pt u \in L^{p'}
 \bigl(0,T;(W^{1,p}(\Om))'\bigr),
\quad u(0)=u_{0},\\
    & A(u)=\int_0^u a(s)\ds \in L^p\bigl(0,T;W^{1,p}(\Om)\bigr),\\
    & v\in L^\infty(Q_T)\cap L^r\bigl(0,T;W^{1,r}(\Om)\bigr)\cap
    C\bigl(0,T,
L^r(\Om)\bigr) \quad
\text{for all $r>1$},\\
    & \pt v \in L^{2}\bigl(0,T;(H^1(\Om))'\bigr),\quad v(0)=v_{0},
\end{align*}
and, for all $\varphi \in L^p(0,T;W^{1,p}(\Om))$ and $\psi  \in L^2(0,T;H^1(\Om))$,
\begin{align*}
    & \int_0^T \left \langle \pt u, \varphi \right
    \rangle \dt +\iint_{Q_T} \Bigl\{ |\Grad A(u)|^{p-2}\Grad A(u)
  - \chi u f(u)\Grad v \Bigr\} \cdot \Grad \varphi\dx\dt=0,\\
    & \int_0^T \left \langle \pt v , \psi \right \rangle \dt
    +d\iint_{Q_T}\Grad v \cdot \Grad \psi\dx\dt
    =\iint_{Q_T}g(u,v)\psi\dx\dt.
\end{align*} }
\end{definition}

To ensure, in particular, that all terms and coefficients  are
sufficiently smooth for this  definition to make sense, we require
that $f\in C^1[0,1]$ and $f(1)=0$, and assume that  the diffusion
coefficient
 $a( \cdot)$ has the following properties:
$a\in C^1[0,1]$, $a(0)=a(1)=0$, and $a(s)>0$  for
  $0<s<1$.
 Moreover,  we assume that there exist constants
 $\delta\in (0,1/2)$ and $\gamma_2\geq\gamma_1>1$ such that
\begin{align}\label{bbru_cond-a(u)}
\gamma_1\phi(s)\leq a(s)\leq \gamma_2\phi(s) \quad
 \text{for $s\in[0,\delta]$},\quad
\gamma_1\psi(1-s)\leq a(s)\leq \gamma_2\psi(1-s) \quad
\text{for $s\in[1-\delta,1]$,}
\end{align}
where we define the functions
$\smash{\phi(s):=s^{\beta_1/(p-1)}}$ and
 $\smash{\psi(s):=s^{\beta_2/(p-1)}}$ for
$\beta_2>\beta_1>0$.

Our first main result
is the following existence theorem for weak solutions.
\begin{theorem} \label{bbru_theo-weak} If
$u_{0},v_{0}\in L^\infty(\Om)$ with $0\le u_0\le 1$ and $v_0\ge 0$
a.e. in $\Om$, then there exists a weak solution to the degenerate
system \eqref{bbru_S1-S2-S3} in the sense of Definition \ref{bbru_def1}.
\end{theorem}

In Section~\ref{bbru_Sect:existence}, we
 first  prove the existence of solutions to a regularized
 version of  \eqref{bbru_S1-S2-S3} by
applying  the Schauder
fixed-point theorem. The regularization basically consists
 in replacing the degenerate diffusion coefficient $a(u)$  by
 the regularized, strictly positive diffusion coefficient
 $a_{\varepsilon} (u)  := a(u) + \varepsilon$,
 where $\varepsilon >0$ is the regularization parameter.
Once the regularized
 problem is solved, we send the regularization parameter~$\varepsilon$
to zero to produce a weak solution of the original system
\eqref{bbru_S1-S2-S3} as the limit of a sequence of such
approximate solutions. Convergence is proved by means of
\textit{a priori} estimates and compactness arguments.

We denote by $\pt Q_T$ the parabolic boundary of $Q_T$,
define $\smash{\tilde{M}:=\|u\|_{\infty,Q_T}}$, and recall the definition of
the intrinsic parabolic $p$-distance from a compact set $K\subset Q_T$ to $\pt Q_T$ as
$$\pdist(K;\pt Q_T):=\inf_{(x,t)\in K,\ (y,s)\in\pt Q_T}\bigl(|x-y|+
\tilde{M}^{(p-2)/p}|t-s|^{1/p}\bigr).$$
Our second main result is the interior local H\"older regularity of weak solutions.
\begin{theorem}\label{bbru_thm-holder}
Let $u$ be a bounded local weak solution of \eqref{bbru_S1-S2-S3} in the
sense of
 Definition~\ref{bbru_def1}, and $\tilde{M}=\|u\|_{\infty,Q_T}$.
 Then $u$ is locally H\"older continuous in $Q_T$, i.e., there exist constants
 $\gamma>1$ and  $\alpha\in(0,1)$, depending only
on the data, such that, for every compact $K\subset Q_T$,
\begin{align*}
\bigl|u(x_1,t_1)-u(x_2,t_2)\bigr|\leq\gamma
\tilde{M}\biggl\{\frac{|x_1-x_2|+ \tilde{M}^{(p-2)/p}
|t_2-t_1|^{1/p}}{ \pdist (K;\pt Q_T)}\biggr\}^\alpha,  \qquad
\forall (x_1,t_1), (x_2,t_2) \in K.
\end{align*}
\end{theorem}

In Section~\ref{bbru_Sect:holdercont}, we prove
Theorem~\ref{bbru_thm-holder} using the method of intrinsic scaling.
This technique is based on analyzing the underlying PDE in a
geometry dictated by its own degenerate structure, that amounts,
roughly speaking, to accommodate its degeneracies. This is achieved
by rescaling the standard parabolic cylinders by a factor that
depends on the particular form of the degeneracies and on the
oscillation of the solution, and which allows for a recovery of
homogeneity. The crucial point is  the proper choice of the
intrinsic geometry which, in the case studied here, needs to take
into account the $p$-Laplacian structure of the diffusion term, as
well as the fact that the diffusion coefficient $a(u)$ vanishes at
$u=0$ and $u=1$. At the core  of the  proof is  the study of an
alternative, now a standard type of argument \cite{DiBe}.
 In either case the conclusion is that when going from
a rescaled cylinder into a smaller one, the oscillation of the solution decreases
in a way that  can be quantified.

In the statement of Theorem~\ref{bbru_thm-holder} and its proof,
 we focus on the interior regularity of~$u$; that of~$v$ follows from
 classical theory of parabolic PDEs \cite{Ladyzenskaia:1968}.
 Moreover, standard adaptations of the method are sufficient
 to extend the results to the parabolic boundary, see \cite{DiBe,DiB:current}.

\subsection{Outline}
The remainder  of the paper is organized as follows: Section
\ref{bbru_Sect:existence} deals with the general proof of our first
main result (Theorem \ref{bbru_theo-weak}).
Section~\ref{bbru_Sect:nondegenerate} is devoted to the detailed
proof of existence of solutions to a non-degenerate problem; in
Section~\ref{bbru_Sect:fixed} we state and prove a fixed-point-type
lemma, and the conclusion of the proof of Theorem
\ref{bbru_theo-weak} is contained in Section
\ref{bbru_Sect:proof-weak}. In Section~\ref{bbru_Sect:holdercont} we
use the method of intrinsic scaling to prove Theorem
\ref{bbru_thm-holder}, establishing the H\"older continuity of weak
solutions to \eqref{bbru_S1-S2-S3}. Finally, in 
 Section~\ref{bbru_Sec:example} we present two  numerical examples  showing the
effects  of prevention of overcrowding and of 
 including the $p$-Laplacian term, and in the Appendix we
give further details about the numerical method used to treat the
examples.

\section{Existence of solutions} \label{bbru_Sect:existence}

We first prove the existence of solutions to a non-degenerate,
regularized version of
  problem~\eqref{bbru_S1-S2-S3},  using the Schauder
fixed-point theorem, and our approach closely follows that of
 \cite{bku}. We define the  following
closed subset of the Banach space $L^p(Q_T)$:
\begin{align*}
\K:=\bigl\{u\in L^p(Q_T) \, : \,  0\le u(x,t)\le 1 \; \text{for a.e.
 $(x,t)\in Q_T$}  \bigr\}.
\end{align*}

\subsection{Weak solution to a non-degenerate problem}\label{bbru_Sect:nondegenerate}
We define  the new diffusion term  $A_\eps(s):=A(s)+\eps s$, with
$a_\eps(s)=a(s)+\eps$,  and consider, for each fixed $\eps>0$, the
non-degenerate  problem
\begin{subequations}
    \label{bbru_S-reg}
    \begin{align}
        & \pt u_\eps -\Div\bigl(|\Grad A_\eps(u_\eps)|^{p-2}\Grad A_\eps(u_\eps)\bigr)
+\Div\bigl(\chi f(u_\eps)\Grad v_\eps\bigr)=0
\quad \text{in $Q_{T}$},  \label{bbru_eq2.1a}  \\
        & \pt v_\eps - d\Delta v_\eps
=g(u_\eps,v_\eps)  \quad \text{in $Q_{T}$}, \label{bbru_eq2.1b} \\
        & |\Grad A_\eps(u_\eps)|^{p-2}a_\eps(u_\eps)
\frac{\partial u_\eps}{\partial \eta}=0, \label{bbru_eq2.1c} \quad
        \frac{\partial v_\eps}{\partial \eta}=0
\quad \text{on $\Sigma_T$},\\
        &u_\eps(x,0)=u_0(x),\quad v_\eps(x,0)=v_0(x)
\quad \text{for $x \in \Om$}.\label{bbru_eq2.1d}
    \end{align}
\end{subequations}
With $\uu \in \K$ fixed, let $v_{\varepsilon}$ be the unique
solution of the
 problem
\begin{subequations}
\label{bbru_S1-v}
  \begin{align}
&  \pt v_{\varepsilon}  - d \Delta v_{\varepsilon}=g(\uu,v_{\varepsilon})
\quad \text{in $Q_T$,} \label{bbru_eq2.2a}  \\
& \frac{\partial v_{\varepsilon}}{\partial \eta}=0 \quad \text{on
 $\Sigma_T$},\quad v_{\varepsilon}(x,0)=v_0(x) \quad \text{for $x \in \Omega$}.
\label{bbru_eq2.2b}
  \end{align}
\end{subequations}
Given the function $v_{\varepsilon}$, let $u_{\varepsilon}$ be the
unique solution of the following quasilinear parabolic problem:
\begin{subequations}\label{bbru_S1-u}
  \begin{align}
 & \pt u_{\varepsilon}-
\Div \bigl(|\Grad A_\eps(u_{\varepsilon})|^{p-2}
\Grad A_\eps(u_{\varepsilon})\bigr)+
\Div \bigl(\chi u_{\varepsilon}f(u_{\varepsilon})
\Grad v_{\varepsilon}\bigr)=0 \quad  \text{in
  $Q_T$,} \label{bbru_eq2.3a} \\
&  |\Grad A_\eps(u_{\varepsilon})|^{p-2}a_\eps(u_{\varepsilon})
\frac{\partial u_{\varepsilon}}{\partial \eta}=0 \quad
\text{on  $\Sigma_T$,}\quad u_{\varepsilon}(x,0)=u_0(x) \quad
\text{for $x \in \Omega$}.  \label{bbru_eq2.3b}
\end{align}
\end{subequations}
Here $v_0$ and~$u_0$ are functions satisfying the assumptions  of
Theorem~\ref{bbru_theo-weak}.

Since for any fixed $\uu \in \K$, \eqref{bbru_eq2.2a} is uniformly
parabolic,  standard theory for parabolic equations
 \cite{Ladyzenskaia:1968}
 immediately leads to the following lemma.
\begin{lemma}\label{bbru_unlema}
If  $v_0 \in L^\infty(\Om)$, then problem \eqref{bbru_S1-v} has a unique weak
solution $v_{\varepsilon} \in L^\infty(Q_T) \cap L^r(0,T;W^{2,r}(\Om)) \cap
C(0,T;L^r(\Om))$, for all $r>1$, satisfying in particular
\begin{align}\label{bbru_est-v-calssic}
\| v_{\varepsilon} \|_{L^\infty(Q_T)}
+\|v_{\varepsilon}\|_{L^\infty(0,T;L^2(\Om))}\le C,\quad
\| v_{\varepsilon} \|_{L^2(0,T;H^{1}(\Om))}\le C,\quad
 \|\pt v_{\varepsilon}\|_{L^2(Q_T)}\le C,
\end{align}
where $C>0$ is a constant that  depends only on
$\smash{\norm{v_0}_{L^\infty(\Om)}}$, $\alpha$, $\beta$, and
$\meas{(Q_T)}$.
\end{lemma}

The following lemma (see   \cite{Ladyzenskaia:1968})
 holds for  the quasilinear problem \eqref{bbru_S1-u}.

\begin{lemma}
If  $u_0 \in L^\infty(\Om)$, then, for any $\eps >0$, there exists
a unique weak solution $u_{\varepsilon} \in L^\infty (Q_T)\cap
L^p(0,T;W^{1,p}(\Om))$ to problem \eqref{bbru_S1-u}. \label{bbru_lem2.2}
\end{lemma}

\subsection{The fixed-point method}
\label{bbru_Sect:fixed}
We define  a map $\Theta:\K\to \K$ such that
$\Theta(\uu)=u_{\varepsilon}$, where $u_{\varepsilon}$
solves~\eqref{bbru_S1-u},
 i.e., $\Theta$ is the solution operator of~\eqref{bbru_S1-u}
associated with the coefficient~$\uu$ and the solution
$v_{\varepsilon}$ coming from \eqref{bbru_S1-v}. By using the
Schauder fixed-point theorem, we now prove that $\Theta$ has a fixed
point. First, we need to show that $\Theta$ is continuous. Let
$\smash{\{\uu_n\}_{n \in \mathbb{N}} }$ be a sequence in $\K$ and
$\uu \in \K$ be such that $\uu_n\to \uu$ in $L^p(Q_T)$ as $n \to
\infty$. Define $u_{\varepsilon n}:=\Theta(\uu_n)$,  i.e.,
$u_{\varepsilon  n}$ is the solution of~\eqref{bbru_S1-u} associated
with $\uu_n$ and the solution $v_{\varepsilon n}$ of
\eqref{bbru_S1-v}. To show that $u_{\varepsilon n}  \to \Theta(\uu)$
in $L^p(Q_T)$, we start with the following lemma.

\begin{lemma}\label{bbru_lem-classic:est}
The solutions $u_{\varepsilon n}$ to problem \eqref{bbru_S1-u} satisfy
\begin{itemize}
\item[(i)]    $0\le u_{\varepsilon n}(x,t)\le 1$  for a.e. $(x,t)\in Q_T$.
\item[(ii)]  The sequence
$\smash{\{ u_{\varepsilon n }\}_{n \in \mathbb{N}} }$ is bounded in
$L^p(0,T;W^{1,p}(\Om))\cap L^\infty(0,T;L^2(\Om))$.
\item[(iii)]  The sequence
$\smash{\{ u_{\varepsilon n} \}_{n \in \mathbb{N}} }$  is relatively
compact in $L^p(Q_T)$.
\end{itemize}
\end{lemma}

\begin{proof} The proof follows from that of
 Lemma~2.3 in  \cite{bku} if we take into account that
 $\smash{\{ \partial_t u_{\varepsilon n} \}_{n \in \mathbb{N}}}$ is
 uniformly bounded in $\smash{L^{p'}(0,T;(W^{1,p}(\Om))')}$.
\end{proof}

The following lemma contains a classical result
(see  \cite{Ladyzenskaia:1968}).

\begin{lemma}\label{bbru_lem-classic:est-bis}
There exists a function $v_{\varepsilon } \in L^2(0,T;H^{1}(\Om))$ such that the
sequence $\smash\{v_{\varepsilon n}\}_{n \in \mathbb{N}}$ converges strongly to $v$ in
$L^2(0,T;H^{1}(\Om))$.
\end{lemma}

Lemmas \ref{bbru_lem2.2}--\ref{bbru_lem-classic:est-bis}  imply that there exist
$u_{\varepsilon }\in L^p(0,T;W^{1,p}(\Om))$ and $v_{\varepsilon } \in L^2(0,T;H^{1}(\Om))$ such that,
up to extracting subsequences if necessary,
$u_{\varepsilon n} \to u_{\varepsilon }$ strongly in $L^p(Q_T)$ and
$v_{\varepsilon n} \to v_{\varepsilon }$
 strongly in $L^2(0,T;H^{1}(\Om))$  as $n \to \infty$,
 so $\Theta$ is indeed continuous on $\K$.
Moreover, due to  Lemma \ref{bbru_lem-classic:est}, $\Theta(\K)$ is
bounded in the set
\begin{equation*}
\mathcal{W}:=\bigl\{ u \in L^p\bigl(0,T;W^{1,p}(\Om)\bigr) \, : \,  \pt u \in
L^{p'} \bigl(0,T;(W^{1,p}(\Om))'\bigr) \bigr\}.
\end{equation*}
Similarly to the results of  \cite{R12}, it can be shown
that \mbox{$\mathcal{W} \hookrightarrow L^{p}(Q_T)$} is compact,
and thus $\Theta$ is compact. Now, by the
Schauder fixed point theorem, the operator $\Theta$ has a fixed
point $u_\eps$ such that $\Theta(u_\eps)=u_\eps$. This implies
that there exists a solution ($u_\eps, v_\eps)$ of
\begin{eqnarray}
\int_0^T \left \langle \pt u_\eps, \varphi \right \rangle \dt
+\iint_{Q_T} \Bigl\{ |\Grad A_\eps(u_\eps)|^{p-2}\Grad
A_\eps(u_\eps)-\chi u_\eps f(u_\eps) \Grad v_\eps  \Bigr\} \cdot
\Grad \varphi\dx\dt=0, \nonumber \\
\int_0^T \left \langle \pt v_\eps , \psi \right \rangle \dt +d
\iint_{Q_T} \Grad v_\eps \cdot \Grad \psi \dx\dt =\iint_{Q_T}g(u_\eps,v_\eps)\psi \dx\dt, \label{bbru_form:ueps} \\
\forall \varphi \in L^{p}(0,T;W^{1,p}(\Om))\text{ and }\forall \psi
\in L^{2}(0,T;H^{1}(\Om)). \nonumber
\end{eqnarray}

\subsection{Existence of weak solutions}
\label{bbru_Sect:proof-weak}
We now
pass to the
limit $\eps \to 0$  in
solutions  $(u_{\eps},v_{\eps})$ to obtain weak solutions of the original system
\eqref{bbru_S1-S2-S3}. From the previous lemmas and
  considering  \eqref{bbru_eq2.1b}, we obtain  the
following result.
\begin{lemma}\label{bbru_maxprinc-est_v_weak}
For each fixed $\eps>0$, the weak solution $(u_\eps,v_\eps)$
to \eqref{bbru_S-reg} satisfies the  maximum principle
\begin{align}\label{bbru_eq:maxprinc}
0\le u_\eps(x,t)\le 1 \quad \text{\em and} \quad v_\eps(x,t) \geq 0
 \quad \text{\em for a.e. $(x,t)\in Q_T$}.
\end{align}
Moreover, the first two estimates of
\eqref{bbru_est-v-calssic} in Lemma~\ref{bbru_unlema} are independent of $\eps$.
\end{lemma}

Lemma \ref{bbru_maxprinc-est_v_weak}  implies that there exists
 a constant  $C>0$, which does not depend on $\eps$, such that
\begin{align}\label{bbru_est-v-weak}
\norm{v_\eps}_{L^\infty(Q_T)}
+\norm{v_\eps}_{L^\infty(0,T;L^2(\Om))}\le C,\quad
&\norm{v_\eps}_{L^2(0,T;H^{1}(\Om))}\le C.
\end{align}
Notice that, from \eqref{bbru_eq:maxprinc} and
\eqref{bbru_est-v-weak}, the term $g(u_\eps,v_\eps)$ is bounded.
Thus, in light of classical results on $L^r$ regularity,  there
exists another constant $C>0$, which is independent of $\eps$, such that
\begin{align*}
\norm{\pt v_\eps}_{L^r(Q_T)}+\norm{v_\eps}_{L^r(0,T;W^{1,r}(\Om))}
\le C \text{ for all $r>1$}.
\end{align*}
Taking $\varphi=A_\eps(u_\eps)$ as a test function in
\eqref{bbru_form:ueps}
yields
\begin{align*}
\int_0^T\langle \pt u_\eps,A(u_\eps)\rangle \dt+
\eps \int_0^T\langle \pt u_\eps,u_\eps\rangle \dt+
\iint_{Q_T}|\Grad A_\eps(u_\eps)|^p \dx\dt
-\iint_{Q_T}\chi f(u_\eps)
\Grad v_\eps\cdot\Grad A_\eps(u_\eps)\dx\dt=0;
\end{align*}
then, using \eqref{bbru_est-v-weak}, the uniform $L^\infty$  bound
on $u_\eps$,  an application of Young's inequality to treat the term
$\Grad v_\eps \cdot \Grad A_\eps(u_\eps)$,  and defining $\smash{
\mathcal{A}_{\varepsilon} (s):=\int_0^s A_{\varepsilon} (r) \, dr}$, we obtain
\begin{equation}
\label{bbru_est1-u-weak}
\sup_{0 \le t \le T} \int_{\Om}\A_{\varepsilon} (u_\eps)(x,t) \dx
+ \eps \sup_{0 \le t \le T} \int_{\Om}\frac{\abs{u_\eps(x,t)}^2}{2} \dx
+ \iint_{Q_T}|\Grad A_\eps(u_\eps)|^p \dx\dt \le C 
\end{equation}
for some constant $C>0$ independent of $\eps$.

Let $\varphi \in L^p(0,T;W^{1,p}(\Om))$. Using the weak formulation
\eqref{bbru_form:ueps}, \eqref{bbru_est-v-weak}  and
\eqref{bbru_est1-u-weak}, we may follow the reasoning in \cite{bku}
to deduce the bound
\begin{equation}\label{bbru_est2-u-weak}
\norm{ \pt u_\eps}_{L^{p'}(0,T;(W^{1,p}(\Om))')} \le C.
\end{equation}
Therefore, from \eqref{bbru_est-v-weak}--\eqref{bbru_est2-u-weak}
and standard compactness
 results (see  \cite{R12}), we can extract subsequences, which we do
 not relabel, such that, as $\eps \to 0$,
 \begin{align}
       \begin{cases}
         A_\eps(u_\eps) \to A(u)
     \text{ strongly in $L^p(Q_T)$ and a.e.},\\
     u_\eps \to u \text{ strongly in $L^q(Q_T)$ for all $q\geq 1$},\\
         v_\eps \to v\text{ strongly in }L^2(Q_T), \\
         \Grad v_\eps \to \Grad v  \text{ weakly in $L^2(Q_T)$}\text{ and } \Grad A_\eps(u_\eps)
     \to \Grad A(u) \text{ weakly in }L^p(Q_T), \\
         |\Grad A_\eps(u_\eps)|^{p-2}\Grad A_\eps(u_\eps) \to \Gamma_1
     \text{ weakly in }L^{p'}(Q_T), \\
         v_\eps \to v \text{ weakly in }L^{2}(0,T;H^{1}(\Om)), \\
         \pt u_\eps \to \pt u \text{ weakly in }L^{p'}(0,T;(W^{1,p}(\Om))')
     \text{ and }\pt v_\eps \to \pt v
     \text{ weakly in }L^{2}(0,T;(H^{1}(\Om))').
         \label{bbru_pa}
         \end{cases}
 \end{align}
To establish the  second convergence in  \eqref{bbru_pa}, we have applied the
dominated convergence theorem to $\smash{u_\eps=
A_\eps^{-1}(A_\eps(u_\eps))}$ (recall that $A$ is monotone) and the
weak-$\star$ convergence of $u_\eps$ to $u$ in $L^\infty(Q_T)$. We
also have the following lemma, see \cite{bku} for its proof.

\begin{lemma} \label{bbru_str-con-grad-v}
The  functions  $v_{\eps}$  converge strongly to $v$ in
$L^{2}(0,T;H^{1}(\Om))$ as $\varepsilon \to 0$.
\end{lemma}

Next,  we identify  $\Gamma_1$ as  $|\Grad A(u)|^{p-2}\Grad A(u)$
 when passing  to
the limit $\eps \to 0$ in \eqref{bbru_form:ueps}.
 Due to this particular nonlinearity, we cannot employ the
 monotonicity argument used in \cite{bku}; rather,
 we will   utilize  a
Minty-type  argument \cite{minty} and make repeated use of the
following  ``weak chain rule''
(see e.g.\  \cite{bmp} for a proof).
\begin{lemma}\label{bbru_chain-rule}
Let $b:\R\to\R$ be Lipschitz continuous and nondecreasing. Assume
$u\in L^\infty(Q_T)$ is such that $\pt u\in
L^{p'}(0,T;(W^{1,p}(\Om))')$, $b(u)\in L^p(0,T;W^{1,p}(\Om))$,
$u(x,0) = u_0(x)$ a.e. on $\Omega$,
  with $u_0\in L^\infty(\Om)$. If we define
 $\smash{B(u)=\int_0^ub(\xi)d\xi}$,  then
\begin{align*}
-\int_0^s\left\langle \pt u,b(u)\phi\right\rangle \dt=
\int_0^s\int_{\Om}B(u)\pt\phi \dx\dt
 + \int_{\Om}B(u_0)\phi(x,0)\dx-\int_{\Om}B(u(x,s))\phi(x,s)\dx
\end{align*}
\end{lemma}
\noindent holds for all $\phi \in \D([0,T]\times\Om)$ and for any
$s\in (0,T)$.

\begin{lemma}
There hold
 $\Gamma_1=|\Grad A(u)|^{p-2}\Grad A(u)$ and
 $\Grad A_\eps(u_\eps) \to \Grad A(u)$  strongly in
  $L^p(Q_T)$.
\end{lemma}

\begin{proof} We define
 $\mathcal{Q}_T:= \{ (t,s,x) \, : \,  (x, s) \in Q_t, \; t\in [0, T]\}$.
The first step will be to show that
\begin{align}\label{bbru_liproof-1}
\iiint_{\mathcal{Q}_T}\bigl(\Gamma_1-|\Grad \sigma|^{p-2}\Grad
\sigma \bigr) \cdot\bigl(\Grad A(u)-\Grad\sigma\bigr) \dx\ds\dt \geq
0, \quad \forall \sigma\in L^p(0,T;W^{1,p}(\Om)).
\end{align}
For all fixed $\eps>0$, we have the  decomposition
\begin{align*}
& \iiint_{\mathcal{Q}_T}\bigl(|\Grad A_\eps(u_\eps)|^{p-2}\Grad
A_\eps(u_\eps)-|\Grad \sigma|^{p-2}\Grad \sigma\bigr)
\cdot\bigl(\Grad A(u)-\Grad\sigma\bigr) \dx\ds\dt=I_1+I_2+I_3,\\
&I_1:=   \iiint_{\mathcal{Q}_T}      |\Grad A_\eps(u_\eps)|^{p-2}\Grad
A_\eps(u_\eps)\cdot\bigl(\Grad A(u)-\Grad A_\eps(u_\eps)\bigr) \dx\ds\dt,\\
&I_2:=   \iiint_{\mathcal{Q}_T}      \bigl(|\Grad A_\eps(u_\eps)|^{p-2}\Grad
A_\eps(u_\eps)-|\Grad \sigma|^{p-2}\Grad \sigma\bigr)
\cdot\bigl(\Grad A_\eps(u_\eps)-\Grad\sigma\bigr) \dx\ds\dt,\\
&I_3 :=   \iiint_{\mathcal{Q}_T}         |\Grad \sigma|^{p-2}\Grad \sigma\cdot
\bigl(\Grad A_\eps(u_\eps)-\Grad A(u)\bigr) \dx\ds\dt.
\end{align*}
Clearly,  $I_2\geq 0$ and from \eqref{bbru_pa} we   deduce that
$I_3 \to 0$ as $\varepsilon \to 0$. For $I_1$, if we
multiply  \eqref{bbru_eq2.1a} by
$\phi\in L^p(0,T;W^{1,p}(\Om))$ and integrate over
$\mathcal{Q}_T$,  
 we obtain
\begin{align*}
\int_0^T\int_0^t \langle \pt u_\eps,\phi \rangle \ds\dt
   -\iiint_{\mathcal{Q}_T} \chi u_\eps f(u_\eps)\Grad
   v_\eps\cdot\Grad \phi \dx\ds\dt
 + \iiint_{\mathcal{Q}_T}  |\Grad A_\eps(u_\eps)|^{p-2}\Grad
   A_\eps(u_\eps)\cdot\Grad \phi \dx\ds\dt=0.
\end{align*}
Now, if we take $\phi=A(u)-A_\eps(u_\eps)\in L^p(0,T;W^{1,p}(\Om))$ and
use Lemma \ref{bbru_chain-rule}, we obtain
\begin{align*}
I_1=&-\int_0^T\int_0^t\left\langle \pt u_\eps,A(u)\right\rangle \ds\dt+
 \int_0^T\int_0^t\left\langle \pt u_\eps,A_\eps(u_\eps)\right\rangle
\, ds \, dt
\\
& + \iiint_{\mathcal{Q}_T} \chi u_\eps f(u_\eps)
 \Grad v_\eps\cdot\bigl(\Grad A(u)-\Grad A_\eps(u_\eps) \bigr)
 \dx\ds\dt  \\
=&-\int_0^T\int_0^t\left\langle \pt u_\eps,A(u)\right\rangle \ds\dt+
  \iint_{Q_T}\A_\eps(u_\eps) \dx\dt-T\int_{\Om}
  \A_\eps(u_0) \dx\\
 & + \iiint_{\mathcal{Q}_T}\chi u_\eps f(u_\eps)
 \Grad v_\eps\cdot\bigl(\Grad A(u)-\Grad A_\eps(u_\eps) \bigr) \dx\ds\dt.
\end{align*}
Therefore, using \eqref{bbru_pa} and Lemma \ref{bbru_str-con-grad-v} 
 and defining $\smash{\mathcal{A} (u) := \int_0^u A(s) \, ds}$, 
we
conclude  that
\begin{equation*}
\lim_{\eps\to 0} I_1=-\int_0^T\int_0^t\left\langle \pt u,A(u)\right\rangle
\ds\dt + \int_0^T\int_{\Om}\A(u(x,t)) \dx\dt -T\int_{\Om}
  \A(u_0(x)) \dx,
\end{equation*}
and from Lemma \ref{bbru_chain-rule}, this yields $I_1\to0$ as $\varepsilon
\to 0$. Consequently,  we have shown that
\begin{align*}
\lim_{\eps\to 0}\iiint_{\mathcal{Q}_T} \bigl(|\Grad A_\eps(u_\eps)|^{p-2}\Grad
A_\eps(u_\eps)-|\Grad \sigma|^{p-2}\Grad \sigma\bigr)
\cdot\bigl(\Grad A(u)-\Grad \sigma\bigr) \dx\ds\dt \geq 0,
\end{align*}
which proves \eqref{bbru_liproof-1}. Choosing
$\sigma=A(u)-\lambda\xi$  with $\lambda\in \R$ and $\xi$
$\in L^p(0,T;W^{1,p}(\Om))$ and  combining  the two inequalities
 arising
from  $\lambda>0$ and $\lambda<0$, we  obtain
 the first assertion of the lemma.
  The second assertion
  directly follows from
\eqref{bbru_liproof-1}.
\end{proof}

With the above convergences we are now able to pass to the limit
$\varepsilon \to 0$, and we can identify the limit
$(u,v)$ as a (weak) solution of \eqref{bbru_S1-S2-S3}. In fact, if  $\varphi \in
L^p(0,T;W^{1,p}(\Om))$ is a test function for \eqref{bbru_form:ueps},
then by
\eqref{bbru_pa} it is now clear  that
\begin{gather*}
\int_0^T \left \langle \pt u_\eps,\varphi \right \rangle \dt
\to \int_0^T \left \langle \pt u,\varphi\right \rangle \dt
\quad \text{as $\eps \to 0$,} \\
\iint_{Q_T} |\Grad A_\eps(u_\eps)|^{p-2}\Grad A_\eps(u_\eps)
\cdot\Grad \varphi \dx \dt
\to \iint_{Q_T}|\Grad A(u)|^{p-2}\Grad A(u)\cdot \Grad
\varphi \dx \dt\quad \text{as $\eps \to 0$.}
\end{gather*}
Since $h(u_\eps)=u_\eps f(u_\eps)$ is bounded in $L^\infty(Q_T)$
 and by Lemma \ref{bbru_str-con-grad-v}, $v_\eps \to
v$ in $L^2(0,T;H^{1}(\Om))$, it follows that
\begin{align*}
\iint_{Q_T} \chi u_\eps f(u_\eps)
\Grad v_\eps \cdot \Grad \varphi \dx \dt
\to \iint_{Q_T} \chi u f(u) \Grad v
\cdot \Grad \varphi \dx \dt \quad \text{as $\eps \to 0$.}
\end{align*}
We have thus identified $u$ as the first component of a solution
of \eqref{bbru_S1-S2-S3}. Using a similar argument, we can identify
$v$ as the second component of a solution.

\section{H\"older continuity of weak solutions}
\label{bbru_Sect:holdercont}

\subsection{Preliminaries}\label{bbru_subsec:holder-prelim}
We start by recasting Definition~\ref{bbru_def1} in a form that
involves the Steklov average, defined for a function $w\in L^1(Q_T)$
and $0<h<T$  by
\begin{align*}
w_h   :=\begin{cases}
{\displaystyle \frac{1}{h}\int_t^{t+h}w(\cdot,\tau)\,d\tau}  & \text{if $t\in(0,T-h]$,}\\
0&\text{if $t\in (T-h,T]$.}
\end{cases}
\end{align*}

\begin{definition}
A local weak solution for \eqref{bbru_S1-S2-S3} is a measurable
function $u$ such that, for every compact $K\subset \Om$ and for all
$0<t<T-h$,
\begin{align}\label{bbru_weak-steklov}
\int_{K\times\{t\}}\Bigl\{\pt(u_h)\varphi+ \bigl(|\Grad
A(u)|^{p-2}\Grad A(u)\bigr)_h\cdot \Grad\varphi-\bigl(\chi
uf(u)\Grad v\bigr)_h\cdot\Grad\varphi\Bigr\}\dx=0,   \quad \forall
\varphi\in W_0^{1,p}(K).
\end{align}
\end{definition}

The following technical lemma on the geometric convergence of
sequences (see e.g., \cite[Lemma~4.2, Ch.~I]{dib1}) will be used later.
\begin{lemma}\label{bbru_geometric}
Let $\{X_n\}$ and $\{Z_n\}$, $n \in \mathbb{N}_0$,  be sequences of positive real numbers satisfying
\begin{align*}
X_{n+1}\leq Cb^n\bigl(X_n^{1+\alpha}+X_n^\alpha Z_n^{1+\kappa}\bigr),
\quad
Z_{n+1}\leq Cb^n \bigl(X_n+Z_n^{1+\kappa}\bigr),
\end{align*}
where $C>1$, $b>1$, $\alpha>0$ and $\kappa>0$ are given constants.
 Then $X_n,Z_n\to 0$ as $n\to \infty$ provided that
\begin{align*}
X_0+Z_0^{1+\kappa}\leq(2C)^{-(1+\kappa)/\sigma}b^{-(1+\kappa)/
  \sigma^2} ,
\quad \text{\em with $\sigma=\min\{\alpha,\kappa\}$.}
\end{align*}
\end{lemma}

\subsection{The rescaled cylinders}
Let $B_\rho(x_0)$ denote the ball of  radius $\rho$ centered at
$x_0$. Then,
for a point $(x_0,t_0)\in \R^{n+1}$, we denote the
cylinder of radius $\rho$ and height $\tau$ by
\begin{align*}
(x_0,t_0)+Q(\tau,\rho):=B_\rho(x_0)\times(t_0-\tau,t_0).
\end{align*}

Intrinsic scaling is based on measuring  the oscillation of a
solution in a family of nested and shrinking cylinders whose
dimensions are  related to the degeneracy of the underlying PDE. To
implement this, we fix $(x_0,t_0)\in Q_T$;  after a translation, we
may assume that $(x_0,t_0)=(0,0)$. Then let $\eps>0$ and let $R>0$
be small enough so that $Q(R^{p-\eps},2R) \subset Q_T$, and define
\begin{align*}
\mu^+:=\esssup_{Q(R^{p-\eps},2R)}u,\qquad  \mu^-:=\essinf_{Q(R^{p-\eps},2R)}u,\qquad
\omega:=\essosc_{Q(R^{p-\eps},2R)}u\equiv\mu^+-\mu^-.
\end{align*}
Now construct the cylinder $Q(a_0R^p,R)$, where
\begin{align*}
a_0=\left(\frac{\omega}{2}\right)^{2-p}\frac{1}{\phi(\omega/2^m)^{p-1}},
\end{align*}
with $m$ to be chosen later. To ensure that
  $Q(a_0R^p,R)\subset Q(R^{p-\eps},2R)$, we
assume that
\begin{equation}\label{bbru_R-e}
\frac{1}{a_0}=\left(\frac{\omega}{2}\right)^{p-2}
\phi\left(\frac{\omega}{2^m}\right)^{p-1}>R^\eps,
\end{equation}
and therefore the relation
\begin{equation}\label{bbru_w--1}
\essosc_{Q(a_0R^{p},R)}u\leq \omega
\end{equation}
holds. Otherwise, the result is trivial as the oscillation is
comparable to the radius. We mention that for $\omega$ small and for
$m> 1$, the cylinder $Q(a_0R^{p},R)$ is long enough in the
$t-$direction, so that we can \textit{accommodate} the degeneracies
of the problem. Without loss of generality, we will assume
$\omega<\delta<1/2$.

Consider now, inside $Q(a_0R^{p},R)$, smaller subcylinders of the form
$$Q_R^{t^*}\equiv(0,t^*)+Q(dR^{p},R),\qquad d=\left(\frac{\omega}{2}\right)^{2-p}
\frac{1}{[\psi(\omega/4)]^{p-1}}, \qquad t^*<0.$$
These are contained in $Q(a_0R^{p},R)$ if  $a_0R^p\geq-t^*+dR^p$, which holds
whenever $\phi(\omega/2^m)\leq\psi(\omega/4)$ and
$$t^*\in\left(\frac{(\omega/2)^{2-p}R^p}{\psi(\omega/4)^{p-1}}-
\frac{(\omega/2)^{p-2}R^p}{\phi(\omega/2^m)^{p-1}},0\right).$$

These particular definitions of $a_0$ and of $d$ turn out to be the
 natural extensions  to the case $p>2$ of their counterparts in
\cite{Urb}. Notice that for $p=2$ and $a(u)\equiv 1$, we recover the
standard parabolic cylinders.

The structure of the proof will be based on the analysis of the
following  alternative: either  there is a cylinder $Q_R^{t^*}$
where $u$ is essentially away from its infimum, or such a cylinder
can not be found and thus $u$ is essentially away from its supremum
in all cylinders of that type. Both cases lead to the conclusion
that the essential oscillation of $u$ within a smaller cylinder
decreases by a factor that can be quantified,  and which does {\em
not} depend on $\omega$.

\begin{rem}\label{bbru_rem:pv}
(See \cite[Remark 4.2]{Porz-Vesp}) Let us introduce quantities of
the type $B_iR^{\theta}\omega^{-b_i}$, where $B_i$ and $b_i>0$ are
constants that can be determined a priori from the data,
independently of $\omega$ and $R$, and $\theta$ depending only on
$N$ and $p$. We assume without loss of generality, that
$$B_iR^{\theta}\omega^{-b_i}\leq 1.$$ If this was not valid,
then we would have $\omega\leq CR^{\eps}$ for the choices
$C=\max_{i}B_i^{1/b}$ and $\eps=\theta/\min_ib_i$, and the
result would be trivial.
\end{rem}

\subsection{The first alternative}
\begin{lemma}\label{bbru_1st-alt-lemma}
There exists $\nu_0\in(0,1)$, independent of $\omega$ and $R$, such
that if
\begin{align}\label{bbru_1st-alt-relation}
\bigl|\bigl\{(x,t)\in Q_R^{t^*} \, : \,  u(x,t)>1-\omega/2 \bigr\}
\bigr|\leq\nu_0\bigl|Q_R^{t^*}\bigr|
\end{align}
for some cylinder
of the type $\smash{Q_R^{t^*}}$,
then $u(x,t)<1-\omega / 4 $ a.e. in $\smash{Q_{R/2}^{t^*}}$.
\end{lemma}
\begin{proof}
Let $u_\omega:= \min \{u,1- \omega/ 4 \}$, take the cylinder for which
\eqref{bbru_1st-alt-relation} holds, define
\begin{align*}
R_n=\frac{R}{2}+\frac{R}{2^{n+1}},\quad n \in \mathbb{N}_0,
\end{align*}
and construct the family
\begin{align*}
Q_{R_n}^{t^*}:=(0,t^*)+Q(dR_n^p,R_n)=B_{R_n}\times(\tau_n,t^*), \quad
\tau_n:=t^*-dR_n^p, \quad n \in \mathbb{N}_0;
\end{align*}
note that $\smash{Q_{R_n}^{t^*}\to Q_{R/2}^{t^*}}$ as $n\to\infty$.
 Let $\{ \xi_n \}_{n \in \mathbb{N}}$ be a sequence
of piecewise smooth cutoff functions satisfying
\begin{equation}\label{bbru_cutoff1}
\begin{cases}
\xi_n=1 \text{ in $Q_{R_{n+1}}^{t^*}$}, \quad
\xi_n=0 \text{ on the parabolic boundary of $Q_{R_n}^{t^*}$},\\[2mm]
\displaystyle
|\Grad \xi_n|\leq\frac{2^{n+1}}{R},\quad 0\leq\pt\xi_n\leq\frac{2^{p(n+1)}}{dR^p},
\quad  |\Delta\xi_n|\leq\frac{2^{p(n+1)}}{R^p},
\end{cases}
\end{equation}
and define
\begin{equation*}
k_n:=1-\frac{\omega}{4}-\frac{\omega}{2^{n+2}},\quad n \in
\mathbb{N}_0.
\end{equation*}
Now take $\smash{\varphi=[(u_\omega)_h-k_n]^+\xi_n^p}$, $K=B_{R_n}$ in \eqref{bbru_weak-steklov}
and integrate in time over $(\tau_n,t)$ for $t\in(\tau_n,t^*)$.
Applying integration by parts to the
first term gives
\begin{align*}
F_1&:= \int_{\tau_n}^t\int_{B_{R_n}}\partial_s
 u_h [(u_\omega)_h-k_n]^+\xi_n^p \, dx \, ds \\
&=  \frac{1}{2}\int_{\tau_n}^t\int_{B_{R_n}}\partial_s
 \Bigl( \bigl([(u_\omega)_h-k_n]^+\bigr)^2 \Bigr) \xi_n^p\, dx \, ds
  +\left(1-\frac{\omega}{4}-k_n\right)\int_{\tau_n}^t\int_{B_{R_n}}\partial_s
 \biggl( \biggl(\Bigl[u-\Bigl(1-
  \frac{\omega}{4}\Bigr) \Bigr]^+\biggr)_h\biggr) \xi_n^p \, dx \, ds \\
&=
\frac{1}{2}\int_{B_{R_n}\times\{t\}}\bigl([u_\omega-k_n]_h^+\bigr)^2\xi_n^p
\, dx \, ds -
  \frac{1}{2}\int_{B_{R_n}\times\{\tau_n\}}\bigl([u_\omega-k_n]_h^+\bigr)^2\xi_n^p
\, dx \, ds
\\
  &\quad  -\frac{p}{2}\int_{\tau_n}^t\int_{B_{R_n}}\bigl([u_\omega-k_n]_h^+\bigr)^2\xi_n^{p-1}
  \partial_s \xi_n   \, dx \, ds  \\ &
 \quad +\left(1-\frac{\omega}{4}-k_n\right)\int_{\tau_n}^t\int_{B_{R_n}}\partial_s
  \biggl( \biggl(\Bigl[u-\Bigl(1-
  \frac{\omega}{4}\Bigr) \Bigr]^+\biggr)_h \biggr) \xi_n^p  \, dx \, ds.
\end{align*}
In light of  standard convergence
properties of the Steklov average, we obtain
\begin{align*}
F_1
\to  F_1^* :=&
\frac{1}{2}\int_{B_{R_n}\times\{t\}}\bigl([u_\omega-k_n]^+\bigr)^2\xi_n^p
 \, dx \, ds
  -\frac{p}{2}\int_{\tau_n}^t\int_{B_{R_n}}\bigl([u_\omega-k_n]^+\bigr)^2\xi_n^{p-1}
  \partial_s \xi_n\, dx \, ds \\
& +\left(1-\frac{\omega}{4}-k_n\right)\biggl(\int_{B_{R_n}\times\{t\}}
  \Bigl[u-\Bigl(1-\frac{\omega}{4}\Bigr)\Bigr]^+\xi_n^{p} \, dx \, ds
\\ & \qquad  -p\int_{B_{R_n}\times\{\tau_n\}}\Bigl[u-\Bigl(1-
  \frac{\omega}{4}\Bigr)\Bigr]^+\xi_n^{p-1}\partial_s  \xi_n \, dx \, ds
  \biggr) \quad \text{as $h \to 0$.}
\end{align*}
 Using \eqref{bbru_cutoff1} and the
nonnegativity of the third term, we arrive at
\begin{align*}
F_1^* & \geq \frac{1}{2}\int_{B_{R_n}\times\{t\}}
\bigl([u_\omega-k_n]^+\bigr)^2\xi_n^p \, dx
-\frac{p}{2d}\left(\frac{\omega}{4}\right)^2\frac{2^{p(n+1)}}{R^p}\int_{\tau_n}^t
\int_{B_{R_n}}\chi_{\{u_\omega\geq k_n\}} \, dx \, ds \\
& \quad-\frac{p}{d}\left(\frac{\omega}{4}\right)^2\frac{2^{p(n+1)}}{R^p}
\int_{\tau_n}^t\int_{B_{R_n}}\chi_{\{u\geq 1-
  \omega/ 4 \}} \, dx \, ds \\
&\geq \frac{1}{2}\int_{B_{R_n}\times\{t\}}\bigl([u_\omega-k_n]^+\bigr)^2
\xi_n^p\, dx   -\frac{3}{2}\frac{p}{d}\left(\frac{\omega}{4}\right)^2
\frac{2^{p(n+1)}}{R^p}\int_{\tau_n}^t\int_{B_{R_n}}
\chi_{\{u_\omega\geq k_n\}} \, dx \, ds,
\end{align*}
the last inequality coming from
$u\geq 1- \omega / 4 \Rightarrow u_\omega\geq k_n$. Since
$\smash{ [u_\omega-k_n]^+\leq \omega / 4}$,  we  know that
\begin{align*}
\bigl([u_\omega-k_n]^+\bigr)^2=\bigl([u_\omega-k_n]^+\bigr)^{2-p}
\bigl([u_\omega-k_n]^+\bigr)^p\geq
\left(\frac{\omega}{4}\right)^{2-p}\bigl([u_\omega-k_n]^+\bigr)^p\geq
\left(\frac{\omega}{2}\right)^{2-p}\bigl([u_\omega-k_n]^+\bigr)^p;
\end{align*}
therefore,  the definition of $d$ implies that
\begin{align}
F_1^*\geq  \frac{1}{2}\left(\frac{\omega}{2}\right)^{2-p}
\int_{B_{R_n}\times\{t\}} \bigl([u_\omega-k_n]^+\bigr)^p\xi_n^p \,
dx -\frac{3}{2}p2^{p-2}\left(\frac{\omega}{4}\right)^p
\frac{2^{p(n+1)}}{R^p}\psi(\omega/4)^{p-1}
\int_{\tau_n}^t\int_{B_{R_n}}\chi_{\{u_\omega\geq k_n\}}\, dx \,
ds.\label{bbru_aux_F1}
\end{align}

We now deal with the diffusive term. The term  
\begin{align*}
F_2:= \int_{\tau_n}^t\int_{B_{R_n}} \bigl(a(u)^{p-1}|\Grad u|^{p-2}\Grad u\bigr)_h\cdot
\Grad\bigl\{[(u_\omega)_h-k_n]^+\xi_n^p\bigr\}\, dx \, ds
\end{align*}
converges for  $h \to 0$  to
\begin{align*}
F_2^* := &  \int_{\tau_n}^t\int_{B_{R_n}}a(u)^{p-1}|\Grad
u|^{p-2}\Grad u\cdot \bigl( \Grad
(u_\omega-k_n)^+\xi_n^p+p(u_\omega-k_n)^+\xi_n^{p-1}\Grad\xi_n
\bigr)
\, dx \, ds \\
=&\int_{\tau_n}^t\int_{B_{R_n}}a(u)^{p-1}\bigl|\xi_n\Grad(u_\omega-k_n)^+\bigr|^p
\, dx \, ds + \tilde{F}_2^*,
\end{align*}
where we define
\begin{align*}
\tilde{F}_2^* :=
p\int_{\tau_n}^t\int_{B_{R_n}}a(u)^{p-1}|\Grad u|^{p-2}\Grad u\cdot
\Grad\xi_n(u_\omega-k_n)^+\xi_n^{p-1} \, dx \, ds.
\end{align*}
Since $\smash{\Grad(u_\omega-k_n)^+}$ is
nonzero only within the set $\{k_n<u<1- \omega/4\}$  and
$$a(u)\geq \gamma_1\psi (\omega / 4 ) \quad \textrm{on} \quad
\{k_n<u<1- \omega / 4 \},$$  we may estimate the first term
 of  $F_2^*$ from below by
\begin{equation}\label{bbru_eq:F2A}
\int_{\tau_n}^t\int_{B_{R_n}}a(u)^{p-1}
\bigl|\xi_n\Grad(u_\omega-k_n)^+\bigr|^p
 \, dx \, ds \geq
\left[\gamma_1\psi(\omega/4)\right]^{p-1}
\int_{\tau_n}^t\int_{B_{R_n}}\bigl|\xi_n\Grad(u_\omega-k_n)^+\bigr|^p\,
dx \, ds.
\end{equation}
Let us now focus on~$\smash{\tilde{F}_2^*}$.
 Using that $\Grad(u_\omega-k_n)^+$ is
nonzero only within the set $\{k_n<u<1-\omega/4 \}$,
integrating by parts, and using \eqref{bbru_cond-a(u)} and \eqref{bbru_cutoff1},
we obtain
\begin{align*}
\bigl| \tilde{F}_2^* \bigr|
& \leq  p\int_{\tau_n}^t\int_{B_{R_n}}|a(u)|^{p-1}\bigl|\Grad
(u_\omega-k_n)^+
 \bigr|^{p-1}
|\Grad\xi_n|(u_\omega-k_n)^+\xi_n^{p-1} \, dx \, ds \\
& \quad +\biggl|p \Bigl(1-\frac{\omega}{4}-k_n\Bigr)\int_{\tau_n}^t\int_{B_{R_n}}\xi_n^{p-1}
\Grad\xi_n\cdot\Grad\biggl\{\frac{1}{p-1}
\biggl(\int_{1-\omega/4}^ua(s)\ds\biggr)
_+^{p-1}\biggr\}\, dx \, ds \biggr|  \\
& \leq  p\left[\gamma_2\psi\left(\omega/ 2 \right)
\right]^{p-1}\int_{\tau_n}^t\int_{B_{R_n}}
|\Grad\xi_n|(u_\omega-k_n)^+\bigl|\xi_n\Grad
(u_\omega-k_n)^+\bigr|^{p-1}
 \, dx \, ds  \\
& \quad  + p \left( \frac{\omega}{4} \right) \biggl|-
\int_{\tau_n}^t\int_{B_{R_n}}\biggl(\int_{1-\omega/4}^ua(s)
\ds\biggr)_+^{p-1}\bigl(
(p-1)\xi_n^{p-2}|\Grad\xi_n|^2+\xi_n^{p-1}\Delta\xi_n
 \bigr) \, dx \, ds
\biggr|.
\end{align*}
Next, we take into account that
\begin{align*}
\biggl(\int_{1-\omega/4 }^ua(s)\ds\biggr)^+\leq
\frac{\omega}{4} \psi\left(\omega / 4 \right),
\end{align*}
and apply  Young's inequality
\begin{align}\label{young}
ab\leq\frac{\epsilon^r}{r}a^r+\frac{b^{r'}}{{r'}\epsilon^{r'}}
\quad\text{if $a, b\geq 0$}, \quad   \frac{1}{r}  +
  \frac{1}{r'} =1, \quad \epsilon >0,
\end{align}
for the choices
\begin{align*}
r=p, \quad
a=|\Grad \xi_n|(u_\omega-k_n)^+,
\quad b=\bigl|\Grad (u_\omega-k_n)^+\bigr|^{p-1}
 \quad \text{and} \quad
\epsilon_1^{-p'}=\frac{p'}{p}\frac{(\gamma_1^{p-1}-1)\psi(\omega/4
)^{p-1}}{\gamma_2^{p-1}\psi(\omega/2)^{p-1}}>0.
\end{align*}
This leads to
\begin{align}
\begin{split}
\bigl| \tilde{F}_2^* \bigr|
& \leq   \frac{1}{\epsilon_1^p}\left[\gamma_2\psi\left(\omega/ 2
\right)\right]^{p-1} \left(\frac{\omega}{4}\right)^p\frac{2^{p(n+1)}}{R^p}
 \int_{\tau_n}^t\int_{B_{R_n}} \chi_{\{u_\omega\geq k_n\}} \, dx \, ds
 \\
& \quad  + (p-1)\epsilon_1^{p'}\left[\gamma_2\psi (\omega/ 2
 )\right]^{p-1}\int_{\tau_n}^t\int_{B_{R_n}}\bigl|\xi_n\Grad
(u_\omega-k_n)^+\bigr|^{p}\, dx \, ds\\ &\quad 
+p^2\left(\frac{\omega}{4}\right)^p
\psi\left(\omega / 4 \right)^{p-1}\frac{2^{p(n+1)}}{R^p}
\int_{\tau_n}^t\int_{B_{R_n}} \chi_{\{u_\omega\geq k_n\}} \, dx \,
ds \\
& \leq  
 \biggl\{\frac{(p-1)\gamma_2^{p-1}\psi\left(\omega/2
   \right)^{p-1}}{(\gamma_1^{p-1}-1)\psi\left(\omega/4 \right)^{p-1}}\biggr\}^{p-1}
\left[\gamma_2\psi(\omega/2)\right]^{p-1}
\left(\frac{\omega}{4}\right)^p\frac{2^{p(n+1)}}{R^p}
 \int_{\tau_n}^t\int_{B_{R_n}} \chi_{\{u_\omega\geq k_n\}}\, dx \, ds \\
 & \quad +\bigl(\gamma_1^{p-1}-1\bigr)\psi\left(\omega/ 4 \right)^{p-1}
 \int_{\tau_n}^t\int_{B_{R_n}}\bigl|\xi_n\Grad
 (u_\omega-k_n)^+\bigr|^{p}
 \, dx \, ds \\
& \quad +p^2\left(\frac{\omega}{4}\right)^p \psi\left(\omega/4
\right)^{p-1}\frac{2^{p(n+1)}}{R^p} \int_{\tau_n}^t\int_{B_{R_n}}
\chi_{\{u_\omega\geq k_n\}} \, dx \, ds.
\end{split}
  \label{bbru_eq:F2B}
\end{align}
Hence, from \eqref{bbru_eq:F2A} and \eqref{bbru_eq:F2B}, and
observing that
\begin{align*}
\left[\frac{\psi\left(\omega/ 2 \right)}{\psi\left(\omega / 4 \right)}
\right]^{p(p-1)}=\left(\frac{4}{2}\right)^{p\beta_2}=2^{p\beta_2},
\end{align*}
we obtain
\begin{align} \label{bbru_aux_F2}
\begin{split}
 F_2^* & \geq \psi\left(\omega/ 4 \right)^{p-1}
\int_{\tau_n}^t\int_{B_{R_n}}\bigl|\xi_n\Grad
(u_\omega-k_n)^+\bigr|^{p} \, dx \, ds   \\
&\quad -\biggl\{p^2+2^{p \beta_2}
\biggl[\frac{{p'}\gamma_2^p}{p(\gamma_1^{p-1}-1)}\biggr]^{p-1} \biggr\}\left(\frac{\omega}{4}\right)^p
 \frac{2^{p(n+1)}}{R^p}
\psi\left(\omega / 4 \right)^{p-1}\int_{\tau_n}^t\int_{B_{R_n}}
 \chi_{\{u_\omega\geq k_n\}}\, dx \, ds .
\end{split}
\end{align}
Finally,  for the lower order term
\begin{align*}
F_3:= \int_{\tau_n}^t\int_{B_{R_n}} \bigl(\chi u f(u)\Grad v\bigr)_h\cdot
\Grad\bigl\{[(u_\omega)_h-k_n]^+\xi_n^p\bigr\} \, dx \, ds
\end{align*}
we have
\begin{align*}
F_3\to F_3^*
 := &\int_{\tau_n}^t\int_{B_{R_n}}\chi u f(u)\Grad v\cdot \bigl( \Grad
(u_\omega-k_n)^+\xi_n^p+p(u_\omega-k_n)^+\xi_n^{p-1}\Grad\xi_n \bigr)
 \, dx \, ds \\
 =&\int_{\tau_n}^t\int_{B_{R_n}}\chi u f(u)\Grad v\cdot \Grad
(u_\omega-k_n)^+\xi_n^p \, dx \, ds
\\  &
+p\int_{\tau_n}^t\int_{B_{R_n}}\chi u f(u)\Grad v\cdot \Grad \xi_n
(u_\omega-k_n)^+\xi_n^{p-1} \, dx \, ds
 \quad \text{as $h \to 0$}.
\end{align*}
Applying Young's inequality \eqref{young}, with
\begin{align*}
r=p,\quad  a=\Grad(u_\omega-k_n)^+\xi_n,\quad
 b=\chi u f(u)\xi_n^{p-1}\Grad v \quad \text{and} \quad
\epsilon_2^p=\frac{p}{2}\psi (\omega / 4 )^{p-1}>0,
\end{align*}
using the fact that $(u_\omega-k_n)^+\leq \omega / 4$ and
defining  $\smash{M:=\|\chi uf(u)\|_{L^\infty(Q_T)}}$, we may
 estimate $F_3^*$ as follows:
\begin{align*}
 F_3^*& \leq \frac{\epsilon_2^p}{p}\int_{\tau_n}^t\int_{B_{R_n}}
\bigl|\Grad(u_\omega-k_n)^+\xi_n\bigr|^p
 \, dx \, ds
+\frac{M^{p'}}{{p'}\epsilon_{2}^{p'}}\int_{\tau_n}^t\int_{B_{R_n}}|\Grad v|^{p'}
\chi_{\{u_\omega\geq k_n\}} \, dx \, ds  \\
&\quad +pM\int_{\tau_n}^t\int_{B_{R_n}} |\Grad v|\left(\frac{\omega}{4}\right)
|\Grad \xi_n|\chi_{\{u_\omega\geq k_n\}} \, dx \, ds \\
& \leq\frac{1}{2} \psi\left(\omega/ 4 \right)^{p-1}
\int_{\tau_n}^t\int_{B_{R_n}}\bigl|\Grad(u_\omega-k_n)^+\xi_n\bigr|^p\,
dx \, ds
+\frac{(p/2)^{-p'/p}}{p'}\frac{M^{p'}}{\psi (\omega/4)}
\int_{\tau_n}^t\int_{B_{R_n}}|\Grad v|^{p'}\chi_{\{u_\omega\geq
  k_n\}}\, dx \, ds \\
 & \quad +\epsilon_3^p\left(\frac{\omega}{4}\right)^p
\int_{\tau_n}^t\int_{B_{R_n}}|\Grad \xi_n|^p\chi_{\{u_\omega\geq
  k_n\}} \, dx \, ds
+\frac{pM^{p'}}{{p'}\epsilon_3^{p'}}
\int_{\tau_n}^t\int_{B_{R_n}}|\Grad v|^{p'}\chi_{\{u_\omega\geq
  k_n\}} \, dx \, ds,
\end{align*}
applying again Young's inequality \eqref{young} to the last term in
the right-hand side, this time with
$$r=p, \quad a=|\Grad\xi_n| \omega / 4 ,\quad  b=M|\Grad v|,\quad
\epsilon_3^{p'}=\psi\left(\omega / 4 \right)>0.$$ Using
\eqref{bbru_cutoff1}, we   obtain
\begin{align*}
F_3^* \leq F_3^{**} :=
& \frac{1}{2}\psi\left(\omega / 4 \right)^{p-1}
\int_{\tau_n}^t\int_{B_{R_n}}\bigl|\Grad(u_\omega-k_n)^+\xi_n\bigr|^p
\, dx \, ds
\\&
+ \frac{M^{p'}}{{p'}\psi (\omega /  4 )}
\left[\left(\frac{p}{2}\right)^{-p'/p}
+p\right]\int_{\tau_n}^t\int_{B_{R_n}}|\Grad v|^{p'}
\chi_{\{u_\omega\geq k_n\}} \, dx \, ds \\
& + \left(\frac{\omega}{4}\right)^p
\frac{2^{p(n+1)}}{R^p}\psi\left( \omega / 4 \right)^{p-1}
 \int_{\tau_n}^t\int_{B_{R_n}}
\chi_{\{u_\omega\geq k_n\}} \, dx \, ds.
\end{align*}
Additionally, using H\"older's inequality, we may write
\begin{align*}
\int_{\tau_n}^t\int_{B_{R_n}}|\Grad v|^{p'}\chi_{\{u_\omega\geq k_n\}}
\, dx \, ds
\leq\|\Grad
v\|^{p'}_{L^{p'p}(Q_T)}\biggl(\int_{\tau_n}^t\bigl|A^+_{k_n,R_n}(\sigma)
\bigr|\,
d\sigma \biggr)^{1-1/p},
\end{align*}
where $\smash{| A^+_{k_n,R_n}(\sigma)|}$ denotes the measure of the
set
\begin{align*}
A^+_{k_n,R_n}(\sigma):=\bigl\{x\in B_{R_n} \, : \, u(x,\sigma)>k_n
 \bigr\}.
\end{align*}
Thus we obtain
\begin{align}\label{bbru_aux_F3}
\begin{split}
 F_3^{**}&  \leq \frac{1}{2}\psi (\omega / 4 )^{p-1}
\int_{\tau_n}^t\int_{B_{R_n}}\bigl|\xi_n\Grad(u_\omega-k_n)^+\bigr|^p
\, dx \, ds +\left(\frac{\omega}{4}\right)^p\frac{2^{p(n+1)}}{R^p}
\psi\left(\omega/ 4 \right)^{p-1}\int_{\tau_n}^t
\int_{B_{R_n}}\chi_{\{u_\omega\geq k_n\}} \, dx \, ds \\
&\quad +\frac{M^{p'}}{{p'}\psi\left(\omega / 4 \right)}
\left[\left(\frac{p}{2}\right)^{-p'/p}
+p\right] \|\Grad v\|^{p'}_{L^{p'p}(Q_T)}
\biggl(\int_{\tau_n}^t\bigl|A^+_{k_n,R_n}(\sigma)\bigr|\,d\sigma\biggr)^{1-1/p}.
\end{split}
\end{align}
Combining   the resulting estimates \eqref{bbru_aux_F1},
\eqref{bbru_aux_F2}, \eqref{bbru_aux_F3} and multiplying  by
$2(\omega/ 2)^{p-2}$ yields
\begin{equation}\label{bbru_aux_ineq}
\begin{split}
&\esssup_{\tau_n\leq t\leq
  t^*}\int_{B_{R_n}\times\{t\}}\bigl([u_\omega-k_n]^+
\bigr)^p\xi_n^p \, dx \, ds
+ \frac{2}{d}\int_{\tau_n}^{t^*}\int_{B_{R_n}}\bigl|\xi_n
\Grad(u_\omega-k_n)^+\bigr|^p \, dx \, ds \\
& \leq  \biggl\{\frac{3}{2}p2^{p-2}+ p^2+2^{p \beta_2}
 \biggl[\frac{{p'}\gamma_2^p}{p(\gamma_1^{p-1}-1)}\biggr]^{p-1}
\biggr\}\left(\frac{\omega}{4}\right)^p\frac{2^{p(n+1)}}{R^p}\frac{2}{d}
\int_{\tau_n}^{t^*}\int_{B_{R_n}}\chi_{\{u_\omega\geq k_n\}} \, dx \,
ds \\
&\qquad +2\frac{\left(\omega/ 2 \right)^{p-2}M^{p'}}{{p'}\psi
\left(\omega / 4 \right)}\left[\left(\frac{p}{2}\right)^{-{p'}/p}
+p\right]\|\Grad v\|^{p'}_{L^{p'p}(Q_T)}
\biggl(\int_{\tau_n}^{t^*}|A^+_{k_n,R_n}(\sigma)|\,d
\sigma\biggr)^{1-1/p}.
\end{split}
\end{equation}
Next we perform a change in the time variable putting
$\smash{\bar{t}:=\frac{1}{d}(t-t^*)}$,
which transforms $\smash{Q(dR_n^p,R_n)}$ into
$\smash{Q^{t^*}_{R_n}}$.
Furthermore, if we define
$\smash{\bar{u}_\omega (\cdot,\bar{t}):=u_\omega(\cdot,t)}$ and
$\smash{\bar{\xi}_n (\cdot,\bar{t})=\xi_n(\cdot,t)}$,
then defining for each $n$,
\begin{align*}
 A_n:=\int_{-R_n^p}^0\int_{B_{R_n}}\chi_{\{\bar{u}_\omega \geq k_n\}}\dx\, d\bar{t}
=\frac{1}{d}\int_{\tau_n}^t\int_{B_{R_n}}\chi_{\{u_\omega\geq k_n\}}
\, dx \, ds
\end{align*}
we may rewrite
\eqref{bbru_aux_ineq} more concisely as
\begin{align} \label{bbru_eq:concise}
\begin{split}
\bigl\|(\bar{u}_\omega -k_n)^+\bar{\xi}_n \bigr\|^p_{V^p(Q_{R_n}^{t^*})}
&\leq
 2\biggl\{\frac{3}{2}p2^{p-2}+ p^2+ 2^{p \beta_2}
 \biggl[\frac{{p'}\gamma_2^p}{p(\gamma_1^{p-1}-1)}\biggr]^{p-1}
\biggr\}\left(\frac{\omega}{4}\right)^p\frac{2^{p(n+1)}}{R^p} A_n \\
& \quad+
2\left[\left(\frac{p}{2}\right)^{-{p'}/p}+p\right]\frac{M^{p'}}{p'}
\left(\frac{\omega}{2}\right)^{(p-2)/p}\psi\left(\omega/4 \right)^{1-p-1/p}\|\Grad
v\|^{p'}_{L^{p'p}(Q_T)}A_n^{1-1/p},
\end{split}
\end{align}
where $V^p(\Om_T)=L^\infty(0,T;L^p(\Om))\cap L^p(0,T;W^{1,p}(\Om))$
endowed with the obvious norm. Next, observe that
 by application of a well-known embedding theorem (cf. \cite[\S I.3]{DiBe}), we get
\begin{align}
\begin{split}
\frac{1}{2^{p(n+1)}}\left(\frac{\omega}{4}\right)^p A_{n+1}
& =   |k_n-k_{n+1}|^p A_{n+1}
 \leq
 \bigl\|(\bar{u}_\omega -k_n)^+\bigr\|^p_{p,Q(R_{n+1}^p,R_{n+1})}
\leq \bigl\|(\bar{u}_\omega -k_n)^+\bar{\xi}_n \bigr\|^p_{p,Q(R_{n}^p,R_{n})}\\
& \leq   C\bigl\|(\bar{u}_\omega -k_n)^+\bar{\xi}_n \bigr\|^p_{V^p(Q_{R_n}^{t^*})}
A_n^{p/(N+p)}.
\end{split}
\label{bbru_eq:in-geo1}
\end{align}
Now, applying  \eqref{bbru_eq:concise}, we get
\begin{align}
\begin{split}
\frac{1}{2^{p(n+1)}}\left(\frac{\omega}{4}\right)^p A_{n+1}& \leq  2C
\biggl\{\frac{3}{2}p2^{p-2}+ p^2+ 2^{p\beta_2}
 \biggl[\frac{{p'}\gamma_2^p}{p
(\gamma_1^{p-1}-1)}\biggr]^{p-1}
\biggr\}\left(\frac{\omega}{4}\right)^p
\frac{2^{p(n+1)}}{R^p} A_n^{1+p/(N+p)}  \\
&\quad +2C\left[\left(\frac{p}{2}\right)^{-q/p}+p\right]\frac{M^{p'}}{p'}
\left(\frac{\omega}{2}\right)^{(p-2)/p}
\psi (\omega/4)^{1-p-1/p}\|\Grad v\|^{p'}_{L^{p'p}(Q_T)}A_n^{1-1/p+
p/(N+p)}.
\end{split}
\label{bbru_eq:in-geo2}
\end{align}
Now let us
define
\begin{align*}
X_n:=\frac{A_n}{|Q(R_n^p,R_n)|},\qquad
Z_n:=\frac{A_n^{1/p}}{|B_{R_n}|}, \quad n \in \mathbb{N}_0.
\end{align*}
Dividing \eqref{bbru_eq:in-geo2} by $\smash{\frac{1}{2^{p(n+1)}}\left(\frac{\omega}{4}
\right)^p|Q(R_{n+1}^p,R_{n+1})|}$ yields
\begin{align*}
X_{n+1}& \leq  2^{pn}\biggl( 2C\biggl\{\frac{3}{2}p2^{p-2}+ p^2
+2^{p\beta_2}
\biggl[\frac{{p'}\gamma_2^p}{p(\gamma_1^{p-1}-1)}\biggr]^{p-1}
  \biggr\} X_n^{1+p/(N+p)}\\
& \quad +
2^{3-2/p+p}C\left[\left(\frac{p}{2}\right)^{-{p'}/p}+p\right]\frac{M^{p'}}{p'}
\left(\frac{\omega}{2}\right)^{p-2}\psi\left(\omega/4\right)^{1-p-1/p}
R^{N\kappa}
\|\Grad v\|^q_{L^{p'p}(Q_T)}X_n^{p/(N+p)}Z_n^{p-1}\biggr)\\
& \leq \gamma 2^{pn}\left(X_n^{1+\alpha}+X_n^\alpha
  Z_n^{1+\kappa}\right),\qquad n
 \in \mathbb{N}_0,
\end{align*}
with $\alpha=p/(N+p)>0$, $\kappa=p-2>0$ and
\begin{align*}
\gamma:=2C\max\biggl\{& \frac{3}{2}p2^{p-2}+ p^2
+2^{p \beta_2}
\biggl[\frac{{p'}\gamma_2^p}{p(\gamma_1^{p-1}-1)}\biggr]^{p-1},  \\
& 2^{3-2/p+p}\left[\left(\frac{p}{2}
\right)^{-{p'}/p}+p\right]\frac{M^{p'}}{p'}
\left(\frac{\omega}{2}\right)^{p-2}
\left[\psi\left(\omega/4 \right)
\right]^{1-p-1/p}R^{N\kappa}\biggr\}>0.
\end{align*}
(In the choice of~$\kappa$  we need
the assumption that~$p$ is {\em strictly} larger than~$2$.)
In the spirit of Remark~\ref{bbru_rem:pv}, let us assume that
\begin{align*}
\left(\frac{\omega}{2}\right)^{p-2}
\left[\psi\left(\omega/4 \right)
\right]^{1-p-1/p}R^{N\kappa}\leq 1.
\end{align*}
Therefore, with this assumption
we conclude that $\gamma$ is independent of $\omega$ and $R$.

Reasoning analogously, we obtain
\begin{align*}
Z_{n+1}\leq \gamma 2^{pn}\left(X_n+Z_n^{1+\kappa}\right).
\end{align*}
Now, let $\sigma=\min\{\alpha,\kappa\}$ and notice that, if we set
$\smash{\nu_0:=2\gamma^{-(1+\kappa)/\sigma}(2^p)^{-(1+\kappa)/\sigma^2}}$,
it follows from \eqref{bbru_1st-alt-relation} that
\begin{align}\label{bbru_ineq:0s}
X_0+Z_0^{1+\kappa}\leq
2\gamma^{-(1+\kappa)/\sigma}(2^p)^{-(1+\kappa)/\sigma^2}.
\end{align}
Then, using Lemma~\ref{bbru_geometric},
we are able to conclude that $X_n,Z_n\to 0$ as $n\to\infty$.
Finally, notice that
$R_n\to R/2$ and $k_n\to 1-\omega/4$, and this implies that
\begin{align*}
\bigl|\bigl\{(x,t)\in Q \bigl((R/2)^p,R/2 \bigr) \, : \,
\bar{u}_\omega (x,\bar{t})\geq 1- \omega / 4 \bigr\}\bigr|=\bigl|
\bigl\{(x,t)\in Q_{R/2}^{t^*} \, : \,  u(x,t)>1-\omega/4 \bigr\}\bigr|=0.
\end{align*}
This completes the proof.
\end{proof}

Now we  show that the conclusion of Lemma~\ref{bbru_1st-alt-lemma}
is valid in a full cylinder of the type $Q(\tau,\rho)$. To this end,
we  exploit the fact that at the time level
$-\hat{t}:=t^*-d(R/2)^p$, the function $x\mapsto u(x,t)$ is strictly
below $1-\omega/4 $ in the ball $B_{R/2}$. We use this time level as
an initial condition to make the conclusion of the lemma hold up to
$t=0$, eventually shrinking the ball. This requires the use of
logarithmic estimates.

Given constants $a,b,c$ with $0<c<a$, we define the nonnegative
function
\begin{align} \label{bbru_rhoplusdef}
\varrho^\pm_{a,b,c}(s)&:= \left(\ln \frac{a}{a+c-(s-b)|_\pm}
 \right)^+ = \begin{cases} \displaystyle
\ln \frac{a}{a+c\pm(b-s)} & \text{if $b\pm c\lessgtr s\lessgtr b\pm (a+c)$,}\\
0& \text{if $s\lesseqgtr b\pm c$,}
\end{cases}
\end{align}
whose first derivative is given by
$$\bigl(\varrho^\pm_{a,b,c}\bigr)'(s)=\begin{cases} \displaystyle
\frac{1}{(b-s)\pm(a+c)}& \text{ if $b\pm c\lessgtr s\lessgtr b\pm (a+c)$} \\
0& \text{ if $s\lessgtr b\pm c$}\end{cases} \quad \gtreqqless 0,$$
and its second derivative, away from  $s=b\pm c$, is
$$\bigl(\varrho^\pm_{a,b,c}\bigr)''=\bigl\{\bigl(\varrho^\pm_{a,b,c}\bigr)'
\bigr\}^2\geq 0.$$ Given $u$ bounded in $(x_0,t_0)+Q(\tau,\rho)$
and a number $k$, define
$$H^\pm_{u,k}:=\esssup_{(x_0,t_0)+Q(\tau,\rho)}\bigl|(u-k)^\pm \bigr|,$$
and the function
\begin{equation}\label{bbru_log-function}
\Psi^\pm\bigl(H^\pm_{u,k},(u-k)^\pm,c\bigr):=
\varrho^\pm_{H^\pm_{u,k},k,c}(u),\qquad 0<c<H^\pm_{u,k}.
\end{equation}

\begin{lemma}
For every number $\nu_1\in(0,1)$, there exists $s_1\in\N$, independent of
$\omega$ and $R$, such that
\begin{equation*}
\bigl|\bigl\{x\in B_{R/4} \, : \,  u(x,t)\geq 1-\omega / 2^{s_1} \bigr\}\bigr|
\leq\nu_1|B_{R/2}| \quad \text{\em for all $t\in(-\hat{t},0)$.}
\end{equation*}
\end{lemma}
\begin{proof}
Let $k=1-\omega/4$ and
\begin{align}
c=\omega/2^{2+n}, \label{cdef}
\end{align}
with
$n\in\N$ to be chosen. Let $0<\zeta(x)\leq 1$ be a piecewise smooth
cutoff function defined on $B_{R/2}$ such that $\zeta=1$ in
$B_{R/4}$ and $|\Grad \zeta|\leq C/R$. Now consider
the weak formulation \eqref{bbru_weak-steklov} with
$\varphi=2\varrho^+(u_h)(\varrho^+)'(u_h)\zeta^p$ for
$K=B_{R/2}$, where $\varrho^+$ is the  function
 defined in \eqref{bbru_rhoplusdef}. After an integration in time over
$(-\hat{t},t)$, with $t\in(-\hat{t},0)$, we obtain
$G_1+G_2-G_3=0$,
where we define
\begin{align*}
G_1 & := 2\int_{-\hat{t}}^t\int_{B_{R/2}}\partial_s \{u_h\}\varrho^+
(u_h)(\varrho^+)'(u_h)\zeta^p \, dx \, ds, \\
G_2 & := 2\int_{-\hat{t}}^t\int_{B_{R/2}}
\bigl(|\Grad A(u)|^{p-2}a(u)\Grad u\bigr)_h\cdot\Grad
\bigl\{ \varrho^+(u_h) (\varrho^+)'(u_h)\zeta^p\bigr\}\, dx \, ds, \\
G_3 & :=
 2\int_{-\hat{t}}^t\int_{B_{R/2}}\bigl(\chi u f(u)\Grad v\bigr)_h
\cdot\Grad\bigl\{\varrho^+(u_h) (\varrho^+)'(u_h)\zeta^p\bigr\}\, dx
\, ds.
\end{align*}
Using the properties of the function $\zeta$,
we arrive at
\begin{align*}
G_1&=\int_{-\hat{t}}^t\int_{B_{R/2}}\partial_s \left\{\varrho^+
(u_h)\right\}^2\zeta^p \, dx \, ds
=\int_{B_{R/2}\times\{t\}}\left\{\varrho^+(u_h)\right\}^2
\zeta^p\dx-\int_{B_{R/2}\times\{-\hat{t}\}}
\left\{\varrho^+(u_h)\right\}^2\zeta^p\dx.
\end{align*}
Due to Lemma~\ref{bbru_1st-alt-lemma}, at time $-\hat{t}$, the
function $x\mapsto u(x,t)$ is strictly below $1-\omega/4$ in the
ball $B_{R/2}$, and therefore $\smash{\varrho^+ (u (x,-\hat{t}) )
=0}$ for $x\in
  B_{R/2}$. Consequently,
\begin{align}  \label{bbru_est-G1}
G_1
\to \int_{B_{R/2}\times\{t\}}\left\{\varrho^+(u)\right\}^2
\zeta^p\dx-\int_{B_{R/2}\times\{-\hat{t}\}}
\left\{\varrho^+(u)\right\}^2\zeta^p\dx
=\int_{B_{R/2}\times\{t\}}\left\{\varrho^+
(u)\right\}^2\zeta^p\dx \quad  \text{as $h \to 0$.}
\end{align}
The definition of $\smash{H^\pm_{u,k}}$ implies that
\begin{equation}\label{bbru_inter-lemma-aux1}
u-k\leq H^+_{u,k}=\esssup_{Q(\hat{t},R/2)}\Bigl|\left(u-1+
\frac{\omega}{4}\right)^+\Bigr|\leq\frac{\omega}{4}.
\end{equation}
If $\smash{H_{u,k}^+=0}$, the result is trivial; so we assume
$\smash{H_{u,k}^+>0}$ and choose $n$ large enough so that
\begin{align*}
0<\frac{\omega}{2^{2+n}}<H_{u,k}^+.
\end{align*}
 Therefore,  since  $\smash{H_{u,k}^++k-u+c>0}$,
the function $\varrho^+(u)$ is defined in the whole cylinder $Q(\hat{t},R/2)$
by
\begin{align*}
\varrho^\pm_{H_{u,k}^+,k,c}(u)=\begin{cases} \displaystyle
\ln \frac{H_{u,k}^+}{H_{u,k}^++c+k-u} & \text{if $u>k+c$,}\\
0& \text{otherwise.}
\end{cases}
\end{align*}
Relation \eqref{bbru_inter-lemma-aux1} implies that
\begin{equation}\label{bbru_inter-aux-log1}
\frac{H_{u,k}^+}{H_{u,k}^++c+k-u}\leq\frac{\frac{\omega}{4}}{2c-
\frac{\omega}{4}}=2^n, \  \text{ and therefore }
\varrho^+(u)\leq n\ln 2;
\end{equation}
in the nontrivial case $u>k+c$, we also have an estimate for
the derivative of the logarithmic function:
\begin{equation}\label{bbru_inter-aux-log2}
\bigl|(\varrho^+)'(u)\bigr|^{2-p}=\biggl|\frac{-1}{H_{u,k}^++c+k-u}
\biggr|^{2-p}\leq\left|\frac{1}{c}\right|^{2-p}=\left(\frac{\omega}{2^{2+n}}
\right)^{p-2}.
\end{equation}
With these estimates at hand, we have for the diffusive term:
\begin{align*}
G_2 \to G_2^*
&:=  2\int_{-\hat{t}}^t\int_{B_{R/2}}a(u)^{p-1}
|\Grad u|^{p-2}\Grad u\cdot\Grad\bigl\{
\varrho^+(u) (\varrho^+)'(u)\zeta^p\bigr\}\, dx \, ds  \\
&=\int_{-\hat{t}}^t\int_{B_{R/2}}a(u)^{p-1}|\Grad u|^{p}
\bigl\{2 \bigl(1+\varrho^+(u)\bigr)\left[(
\varrho^+)'(u)\right]^2\zeta^p\bigr\}  \, dx \, ds  +  \tilde{G}_2^*
 \quad \text{as $h \to 0$,}
\end{align*}
where we define
\begin{align*}
\tilde{G}_2^* := 2p\int_{-\hat{t}}^t\int_{B_{R/2}}a(u)^{p-1} |\Grad
u|^{p-2}\Grad u\cdot\Grad\zeta \bigl\{\varrho^+(u)
(\varrho^+)'(u)\zeta^{p-1}\bigr\} \, dx \, dt.
\end{align*}
Applying Young's inequality \eqref{young} with the choices
\begin{align*}
r=p, \quad a=|\Grad u|^{p-1} \zeta^{p-1} \bigl|(\varrho^+)'(u)
\bigr|^{2/p'}, \quad b=\bigl|(\varrho^+)'(u)
\bigr|^{1-2/p'}|\Grad\zeta| \quad \textrm{and} \quad
\epsilon_4=1,
\end{align*}
we obtain
\begin{align*}
\bigl| \tilde{G}_2^* \bigr|
& \leq 2p\int_{-\hat{t}}^t\int_{B_{R/2}}
a(u)^{p-1}|\Grad u|^{p-1}|\Grad\zeta|
\varrho^+(u) \bigl|(\varrho^+)'(u)\bigr|\zeta^{p-1}\, dx \, ds \\
& = 2p\int_{-\hat{t}}^t\int_{B_{R/2}}
a(u)^{p-1}\varrho^+(u)|\Grad u|^{p-1} \zeta^{p-1}
\bigl|(\varrho^+)'(u)\bigr|^{2/p'}
\bigl|(\varrho^+)'(u) \bigr|^{1-2/p'}|\Grad\zeta| \, dx \, ds \\
& \leq 2\epsilon_4^p\int_{-\hat{t}}^t\int_{B_{R/2}}
a(u)^{p-1}\varrho^+(u)|\Grad u|^p
\bigl[(\varrho^+)'(u)\bigr]^2\zeta^p \, dx \, ds \\
&\quad  +\frac{2p}{p'\epsilon_4^q}\int_{-\hat{t}}^t
\int_{B_{R/2}}a(u)^{p-1} \varrho^+(u)|\Grad \zeta|^p
\bigl|(\varrho^+)'(u)\bigr|^{2-p} \, dx \, ds \\
& = 2\int_{-\hat{t}}^t\int_{B_{R/2}}a(u)^{p-1}
\varrho^+(u)|\Grad u|^p \bigl[(\varrho^+)'(u)\bigr]^2\zeta^p \, dx \,
ds
\\&\quad
+2(p-1)\int_{-\hat{t}}^t\int_{B_{R/2}}
a(u)^{p-1} \varrho^+(u)|\Grad
\zeta|^p\bigl|(\varrho^+)'(u)\bigr|^{2-p} \, dx \, ds.
\end{align*}
In face of this estimate, we obtain
\begin{align*}
 G_2^*& =  2\int_{-\hat{t}}^t\int_{B_{R/2}}a(u)^{p-1}|\Grad u|^p
\bigl[(\varrho^+)'(u)\bigr]^2\zeta^p \, dx \, ds
\\ & \quad
-2(p-1)\int_{-\hat{t}}^t
\int_{B_{R/2}}a(u)^{p-1}\varrho^+(u)|\Grad \zeta|^p\bigl|
 (\varrho^+)' (u)\bigr|^{2-p} \, dx \, ds  \\
& \geq 2\left[\gamma_1
\psi\left(\omega/4 \right)\right]^{p-1}\int_{-\hat{t}}^t
\int_{B_{R/2}}|\Grad u|^p \bigl[ (\varrho^+)'(u)\bigr]^2
\zeta^p \, dx \, ds  \\ &
\quad -2(p-1)\int_{-\hat{t}}^t \int_{B_{R/2}}
a(u)^{p-1}\varrho^+(u)|\Grad \zeta|^p\bigl|(\varrho^+)'
  (u)\bigr|^{2-p}\, dx \, ds  \\
&  \geq2\left[\gamma_1\psi\left(\omega/4 \right)\right]^{p-1}
\int_{-\hat{t}}^t\int_{B_{R/2}}|\Grad u|^p
\bigl[(\varrho^+)'(u)\bigr]^2\zeta^p\, dx \, ds  \\
& \quad   -2(p-1)n\ln
2\left(\frac{C}{R}\right)^p\left(\frac{\omega}{2^{2+n}}
\right)^{p-2}\int_{-\hat{t}}^t\int_{B_{R/2}}a(u)^{p-1} \chi_{\{u>1-
\omega/4\}} \, dx \, ds,
\end{align*}
and, finally,
\begin{align}
\begin{split}
 G_2^* & \geq 2\left[\gamma_1\psi\left(\omega/4 \right)\right]^{p-1}
\int_{-\hat{t}}^t\int_{B_{R/2}}|\Grad u|^p
\bigl[(\varrho^+)'(u)\bigr]^2\zeta^p \, dx \, ds
\\ & \quad
-2(p-1)n\ln 2\left(\frac{C}{R}\right)^p
\left(\frac{\omega}{2^{2+n}} \right)^{p-2}\hat{t}|B_{R/2}|
\left[\gamma_2\psi\left(\omega/4 \right)\right]^{p-1},
\end{split}
\label{bbru_est-G2a}
\end{align}
where we have used estimates \eqref{bbru_inter-aux-log1},
\eqref{bbru_inter-aux-log2}, the properties of $\zeta$, and the fact
that
\begin{align*}
\gamma_1\psi\left(\omega/4\right)\leq a(u)
\leq\gamma_2\psi\left(\omega/4\right) \quad \text{on the set $\{u>1-
\omega/4\}$.}
\end{align*}
Moreover, from the definition of $\hat{t}$ and our choice of $t^*$
(recall that $t^*\geq dR^p-a_0R^p$),  there  holds
\begin{align} \label{bbru_this}
\hat{t}\leq
a_0R^p=\left(\frac{\omega}{2}\right)^{2-p}\frac{R^p}{\phi\left(
\omega / 2^m \right)^{p-1}}.
\end{align}
Taking into account \eqref{bbru_this}, we obtain from
\eqref{bbru_est-G2a} that
\begin{equation}\label{bbru_est-G2}
\begin{split}
 G_2^* & \geq  2\left[\gamma_1\psi\left(\omega/4
\right)\right]^{p-1}\int_{-\hat{t}}^t\int_{B_{R/2}}|\Grad u|^p
\bigl[(\varrho^+)'(u)\bigr]^2\zeta^p\, dx \, ds \\
& \quad  -2(p-1)n\ln 2C^p2^{(1+n)(2-p)}|B_{R/2}|
\left[\gamma_2\frac{\psi\left(\omega/4
\right)}{\phi\left(\omega / 2^{m} \right)} \right]^{p-1}.
\end{split}
\end{equation}
On the other hand,  for the lower order term, by passing to the
limit $h\to 0$, we have
\begin{align*}
G_3
 \to G_3^* & :=  2\int_{-\hat{t}}^t\int_{B_{R/2}}\chi u f(u)
\Grad v\cdot\Grad u\bigl\{
\bigl(1+\varrho^+(u)\bigr)\bigl[(\varrho^+)'(u)\bigr]^2\zeta^p\bigr\}
 \, dx \, ds  \\
& \quad +2p\int_{-\hat{t}}^t\int_{B_{R/2}}\chi u f(u)\Grad v\cdot\Grad
\zeta\bigl\{\varrho^+(u)(\varrho^+)'(u)\zeta^{p-1}\bigr\} \, dx \, ds
 \\
& \leq  2M\int_{-\hat{t}}^t\int_{B_{R/2}} \bigl(1+\varrho^+(u) \bigr)
\bigl[(\varrho^+)'(u)\bigr]^2\zeta^p|\Grad u||\Grad v|\, dx \, ds \\
& \quad +2pM\int_{-\hat{t}}^t\int_{B_{R/2}}\varrho^+(u)
\bigl|(\varrho^+)'(u)\bigr|^{1-2/p'}|\Grad v||\Grad \zeta|
\bigl|(\varrho^+)'(u)\bigr|^{2/p'}\zeta^{p-1} \, dx \, ds .
\end{align*}
Applying Young's inequality \eqref{young} to the first term on the
right-hand side with
\begin{align*}
r=p, \quad  a=|\Grad u|, \quad  b=|\Grad v| \quad \text{and} \quad
 \epsilon_5=\biggl(\frac{p\psi\left(\omega/4\right)^{p-1}}
{M(1+n\ln2)}\biggr)^{1/p},
\end{align*}
and to the second term with
\begin{align*}
r=p, \quad
 a=\bigl| (\varrho^+)'(u)\bigr|^{1-2/p'}, \quad
 b=|\Grad v|\bigl| (\varrho^+)'(u)\bigr|^{2/p'}\zeta^{p-1} \quad
 \text{and} \quad
 \epsilon_6=1,
\end{align*}
we obtain
\begin{align*}
G_3^*
&\leq 2 \psi\left(\omega/4\right)^{p-1}
\int_{-\hat{t}}^t\int_{B_{R/2}}|\Grad u|^p\bigl[(\varrho^+)'(u)
\bigr]^2\zeta^p \, dx \, ds +2M \int_{-\hat{t}}^t\int_{B_{R/2}}\varrho^+(u)|\Grad\zeta|\bigl[(
\varrho^+)'(u)\bigr]^{2-p} \, dx \, ds \\
& \quad +2M\frac{p-1}{p}\biggl(\frac{p\psi\left(\omega/4\right)^{p-1}}
{M(1+n\ln2)}\biggr)^{1/(1-p)}\int_{-\hat{t}}^t\int_{B_{R/2}}\bigl(1+\varrho^+(u)\bigr)
\bigl[(\varrho^+)'(u)\bigr]^2\zeta^{p}|\Grad v|^{p'} \, dx \, ds \\
& \quad +2M(p-1)\int_{-\hat{t}}^t
\int_{B_{R/2}}\varrho^+(u)|\Grad \zeta||\Grad v|^{p'}\bigl[(
\varrho^+)'(u)\bigr]^2\zeta^{p} \, dx \, ds.
\end{align*}
Using  the estimates
\eqref{bbru_inter-aux-log1} and  \eqref{bbru_inter-aux-log2} and
the properties of $\zeta$, we then get
\begin{align*}
G_3^*  &\leq 2 \psi (\omega/4 )^{p-1}
\int_{-\hat{t}}^t\int_{B_{R/2}}|\Grad u|^p\bigl[(\varrho^+)'(u)
\bigr]^2\zeta^p \, dx \, ds +2M n\ln2 \frac{C}{R} \left(\frac{\omega}{2^{2+n}}
\right)^{p-2}\hat{t}|B_{R/2}| \\
& \quad +2M\frac{p-1}{p}\biggl(\frac{p\psi (\omega/4)^{p-1}}
{M(1+n\ln2)}\biggr)^{1/(1-p)}(1+n\ln2)\left(\frac{\omega}{2^{2+n}}
\right)^{-2}\int_{-\hat{t}}^t\int_{B_{R/2}}|\Grad v|^{p'}
\chi_{\{u>1-\omega/4\}} \, dx \, ds \\
& \quad +2M(p-1)n\ln2 \frac{C}{R} \left(
\frac{\omega}{2^{2+n}}\right)^{-2}\int_{-\hat{t}}^t\int_{B_{R/2}}
|\Grad v|^{p'}\chi_{\{u>1-\omega/4\}} \, dx \, ds.
\end{align*}
Then, applying H\"older's inequality and recalling the definition of
$\hat{t}$, we get
\begin{align*}
G_3^*
&\leq 2 \psi (\omega/4 )^{p-1}
\int_{-\hat{t}}^t\int_{B_{R/2}}|\Grad u|^p\bigl[(\varrho^+)'(u)
\bigr]^2\zeta^p \, dx \, ds +2MCn\ln2\,2^{(1+n)(2-p)}
\phi (\omega / 2^{m} )^{1-p}|B_{R/2}|R^{p-1} \\
& \quad + 2M(p-1)\biggl\{ \biggl(\frac{p\psi\left(\omega/4\right)^{p-1}}
{M(1+n\ln2)}\biggr)^{1/(1-p)}\frac{1+n\ln2}{p}+ \frac{C}{R} n\ln2\biggr\}
\left(\frac{\omega}{2^{2+n}}\right)^{-2} \times
\\& \quad \qquad \times
\|\Grad v\|^{p'}_{L^{p'p}(Q_T)}
\left(a_0R^p|B_{R/2}|\right)^{1-1/p}.
\end{align*}
In addition, thanks to Remark \ref{bbru_rem:pv}, we may estimate
\begin{gather*}
 \left(\frac{\omega}{2^{2+n}}\right)^{-2}\biggl(\frac{p^{-p'}
\psi\left(\omega/4\right)^{p-1}}
{M(1+n\ln2)}\biggr)^{1/(1-p)}a_0^{1-1/p} R^{p-1} \leq 1,\\
 C\left(\frac{\omega}{2^{2+n}}\right)^{-2}a_0^{1-1/p}R^{p-2} \leq
1,\quad
\phi\left(\frac{\omega}{2^{m}}\right)^{1-p}R^{p-1} \leq 1,
\end{gather*}
and this finally gives
\begin{align}\label{bbru_est-G3}
\begin{split}
G_3^*
& \leq 2 \psi\left(\omega/4\right)^{p-1}
\int_{-\hat{t}}^t\int_{B_{R/2}}|\Grad u|^p\bigl[(\varrho^+)'(u)
\bigr]^2\zeta^p\, dx \, ds +2MCn\ln2\,2^{(1+n)(2-p)}|B_{R/2}|\\
&\quad  + 2M(p-1)Cn\ln2\|\Grad v\|^{p'}_{L^{p'p}(Q_T)}
|B_{R/2}|^{1-1/p}.
\end{split}
\end{align}

Combining estimates \eqref{bbru_est-G1}, \eqref{bbru_est-G2} and
\eqref{bbru_est-G3} yields
\begin{align*}
 \int_{B_{R/2}\times\{t\}}\bigl\{\varrho^+(u)\bigr\}^2
\zeta^p \, dx \, ds &  \leq  2M(p-1)Cn\ln2\|\Grad v\|^{p'}_{L^{p'p}(Q_T)}
|B_{R/2}|^{1-1/p}\\
&\quad + (1-\gamma_1^{p-1})2\left[\psi\left(\omega/4\right)\right]^{p-1}
\int_{-\hat{t}}^t\int_{B_{R/2}}|\Grad u|^p\bigl[(\varrho^+)'(u)
\bigr]^2\zeta^p\, dx \, ds \\
&\quad  +2n\ln2\,2^{(1+n)(2-p)}|B_{R/2}|
\biggl\{ MC+(p-1)C^p\gamma_2^{p-1}\biggl[\frac{\psi\left(\omega/4
\right)}{\phi\left(\omega/ 2^{m} \right)} \biggr]^{p-1}\biggr\},
\end{align*}
and since $\gamma_1>1$ and  $n>0$, this implies
\begin{align}\label{bbru_swq1}
\begin{split}
 \sup_{-\hat{t}\leq t\leq0}\int_{B_{R/2}\times\{t\}}\bigl\{\varrho^+(u)\bigr\}^2
\zeta^p\, \dx &  \leq   2M(p-1)Cn\ln2\|\Grad v\|^{p'}_{L^{p'p}(Q_T)}
|B_{R/2}|^{1-\frac{1}{p}}\\
&\quad +2n\ln2\,2^{2-p}|B_{R/2}|
\biggl\{MC+(p-1)C^p\gamma_2^{p-1}\biggl[\frac{\psi\left(\omega/4
\right)}{\phi\left(\omega / 2^{m} \right)} \biggr]^{p-1}\biggr\}.
\end{split}
\end{align}
Since the integrand in the left-hand side
 of \eqref{bbru_swq1} is nonnegative,  the integral
  can  be estimated from
below by integrating over the smaller set $S=\{x\in B_{R/2} \, : \,  u(x,t)\geq 1-
\omega / 2^{2+n} \}\subset B_{R/2}$.  Thus, noticing that
\begin{align*}
\zeta=1 \quad \text{and} \quad
\bigl\{\varrho^+(u)\bigr\}^2\geq \bigl( \ln(2^{n-1})\bigr) ^2
=(n-1)^2(\ln 2)^2 \quad \text{on $S$,}
\end{align*}
we obtain that \eqref{bbru_swq1} reads
\begin{align*}
& \bigl|\bigl\{x\in B_{R/2}\, : \,  u(x,t)\geq 1-\omega / 2^{2+n}
\bigr\}\bigr| \\
& \leq
\frac{2Cn|B_{R/4}|}{(n-1)^2\ln2}\biggl\{2^{2-p}\biggl[MC+(p-1)C^p\gamma_2^{p-1}
\biggl[\frac{\psi\left(\omega/4
\right)}{\phi\left(\omega / 2^{m} \right)} \biggr]^{p-1}\biggr]+M(p-1)
\|\Grad v\|^{p'}_{L^{p'p}(Q_T)}\biggr\}
\end{align*}
for all $t\in (-\hat{t},0)$. To prove the lemma we just need to choose $s_1$ depending on
 $\nu_1$ such that $s_1=2+n$ with
 $$n>1+\frac{2C}{\nu_1\ln2}\biggl\{2^{2-p}\biggl[MC+(p-1)C^p\gamma_2^{p-1}
 \biggl[\frac{\psi\left(\omega/4
\right)}{\phi\left(\omega/2^{m} \right)} \biggr]^{p-1}\biggr]+M(p-1)
\|\Grad v\|^{p'}_{L^{p'p}(Q_T)}\biggr\},$$
since if $n\geq1+2 / \alpha$ then $n / (n-1)^2 \leq\alpha$, $\alpha>0$.
Furthermore, $s_1$ is independent of $\omega$ because
$$\biggl[\frac{\psi\left(\omega/4\right)}{\phi\left(\omega / 2^{m} \right)}
\biggr]^{p-1}=\biggl[\frac{\left(\omega / 4  \right)^{\beta_2/(p-1)}}
{\left(\omega / 2^{m} \right)^{\beta_1 / (p-1)}}\biggr]^{(p-1)}=\omega^{\beta_2-\beta_1}
2^{m\beta_1-2\beta_2}\leq 2^{m\beta_1-2\beta_2}.$$
The  last inequality holds since  $\beta_2>\beta_1$.
\end{proof}
Now, the first alternative is established by the following proposition.
\begin{proposition}\label{bbru_reduction-osc}
The numbers $\nu_1\in(0,1)$ and $s_1\gg 1$ can be chosen a priori
independently of $\omega$ and $R$, such that if
\eqref{bbru_1st-alt-relation} holds, then
\begin{align*}
u(x,t)<\frac{\omega}{2^{s_1+1}} \quad
\text{\em  a.e. in  $Q(\hat{t},R/8)$.}
\end{align*}
\end{proposition}
We omit the proof of Proposition \ref{bbru_reduction-osc} because it is based on the
argument of \cite[Lemma 3.3]{DiBe} and \cite{Urb}, and we may use 
  for the
extension  the same
technique applied in the proof of Lemma~\ref{bbru_1st-alt-lemma}.
\begin{corollary}\label{bbru_coro-1st}
There exist numbers $\nu_0,\sigma_0\in(0,1)$
independent of $\omega$ and $R$ such that if \eqref{bbru_1st-alt-relation} holds,
then
$$\essosc_{Q(\hat{t},R/8)} u\leq \sigma_0 \omega.$$
\end{corollary}
\begin{proof}
In light of Proposition \ref{bbru_reduction-osc}, we know that there
exists a number  $s_1$ such that
$$\esssup_{Q(\hat{t},R/8)}u\leq 1 -\frac{\omega}{2^{s_1+1}},$$
and this yields
$$\essosc_{Q(\hat{t},R/8)}u=\esssup_{Q(\hat{t},R/8)}u-
\essinf_{Q(\hat{t},R/8)}u\leq \left(1
-\frac{1}{2^{s_1+1}}\right)\omega.$$ In this way, choosing
$\sigma_0=1 -1/2^{s_1+1}$, which is independent of $\omega$, we
complete the proof.
\end{proof}
\subsection{The second alternative}
Let us suppose now that \eqref{bbru_1st-alt-relation} does not hold. Then the complementary case
is valid and for every cylinder $Q_R^{t^*}$ we have
\begin{equation}\label{bbru_2nd-alt-relation}
\bigl|\bigl\{(x,t)\in Q_R^{t^*} \,  : \,   u(x,t)<\omega/ 2 \bigr\}\bigr|\leq
(1-\nu_0)\bigl|Q_R^{t^*}\bigr|.
\end{equation}
Following an analogous analysis to the performed in the case in
which the solution is near its degeneracy at one, a similar
conclusion is obtained for the second alternative (cf. \cite{bku}
and \cite{Urb}). Specifically, we first use logarithmic estimates to
extend the result to a full cylinder and then we conclude that the
solution is essentially away from 0 in a cylinder $Q(\tau,\rho)$. In
this way we prove  the following corollary.
\begin{corollary}\label{bbru_coro-2nd}
Let  $\tilde{t}$ denote the second-alternative-counterpart of $\hat{t}$.
Then there exists $\sigma_1\in(0,1)$, depending only on the data, such that
\begin{align*}
\essosc_{Q(\tilde{t},R/8)}u\leq\sigma_1\omega.
\end{align*}
\end{corollary}
Since \eqref{bbru_1st-alt-relation} or \eqref{bbru_2nd-alt-relation}
must be valid,
the conclusion of Corollary \ref{bbru_coro-1st} or \ref{bbru_coro-2nd}
must
hold. Thus, choosing
$\sigma=\max\{\sigma_0,\sigma_1\}$ and $
t^\diamond=\min\{\hat{t},\tilde{t}\}$,
 we obtain the following proposition.
\begin{proposition}
There exists a constant $\sigma\in(0,1)$, depending only on the data, such that
$$\essosc_{Q\left(t^\diamond,R/8\right)}u\leq\sigma\omega.$$
\end{proposition}
The local H\"older continuity of $u$ in $Q_T$ now follows
(see, e.g., \cite{DiBe}, \cite{Urb-book}, or the proof of
\cite[Th. 2]{DiB:current}).
\section{Numerical examples}\label{bbru_Sec:example}
\begin{figure}[t]
\begin{center}
\begin{tabular}{cc}
\includegraphics[width=0.35\textwidth]{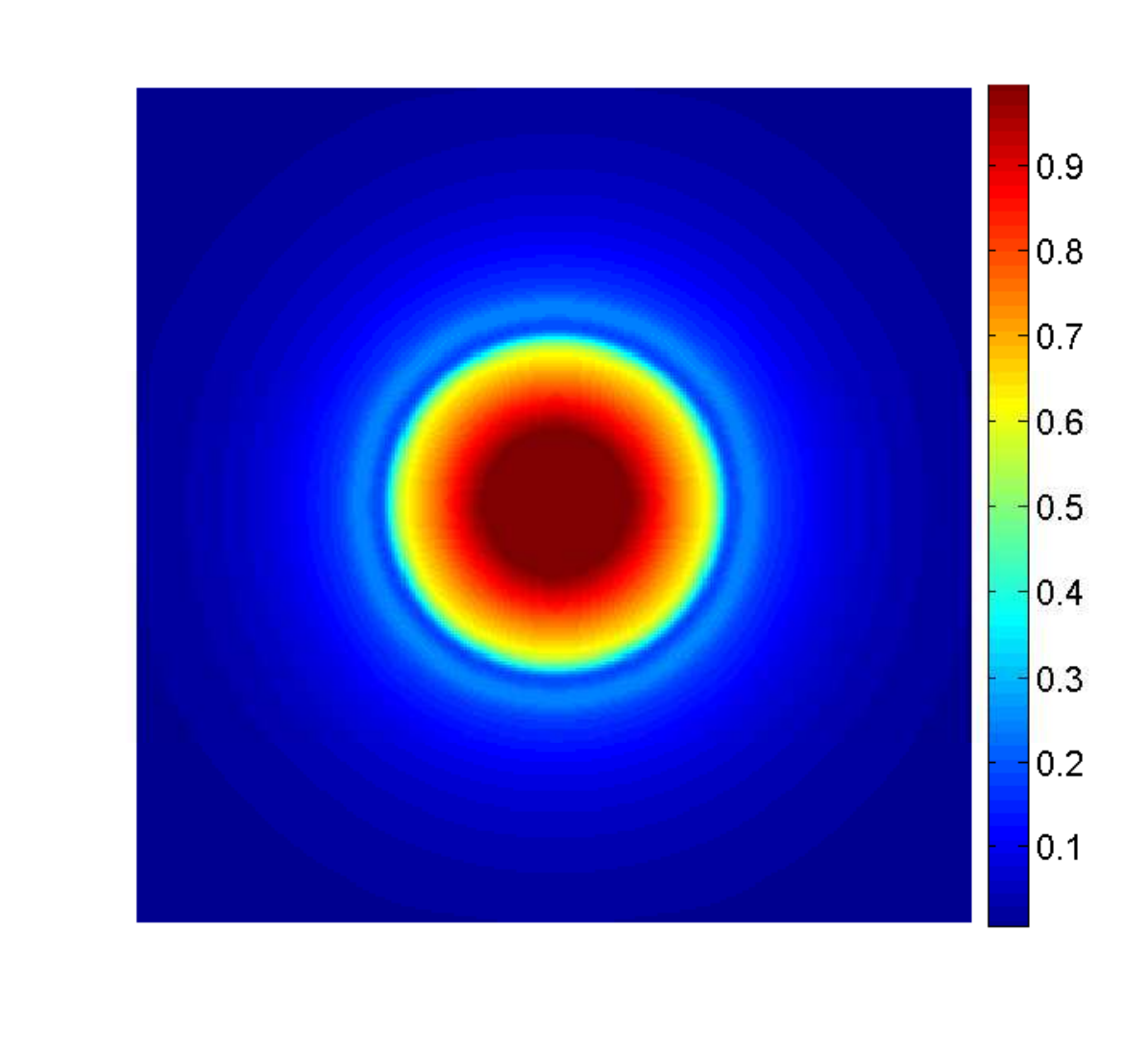}&
\includegraphics[width=0.35\textwidth]{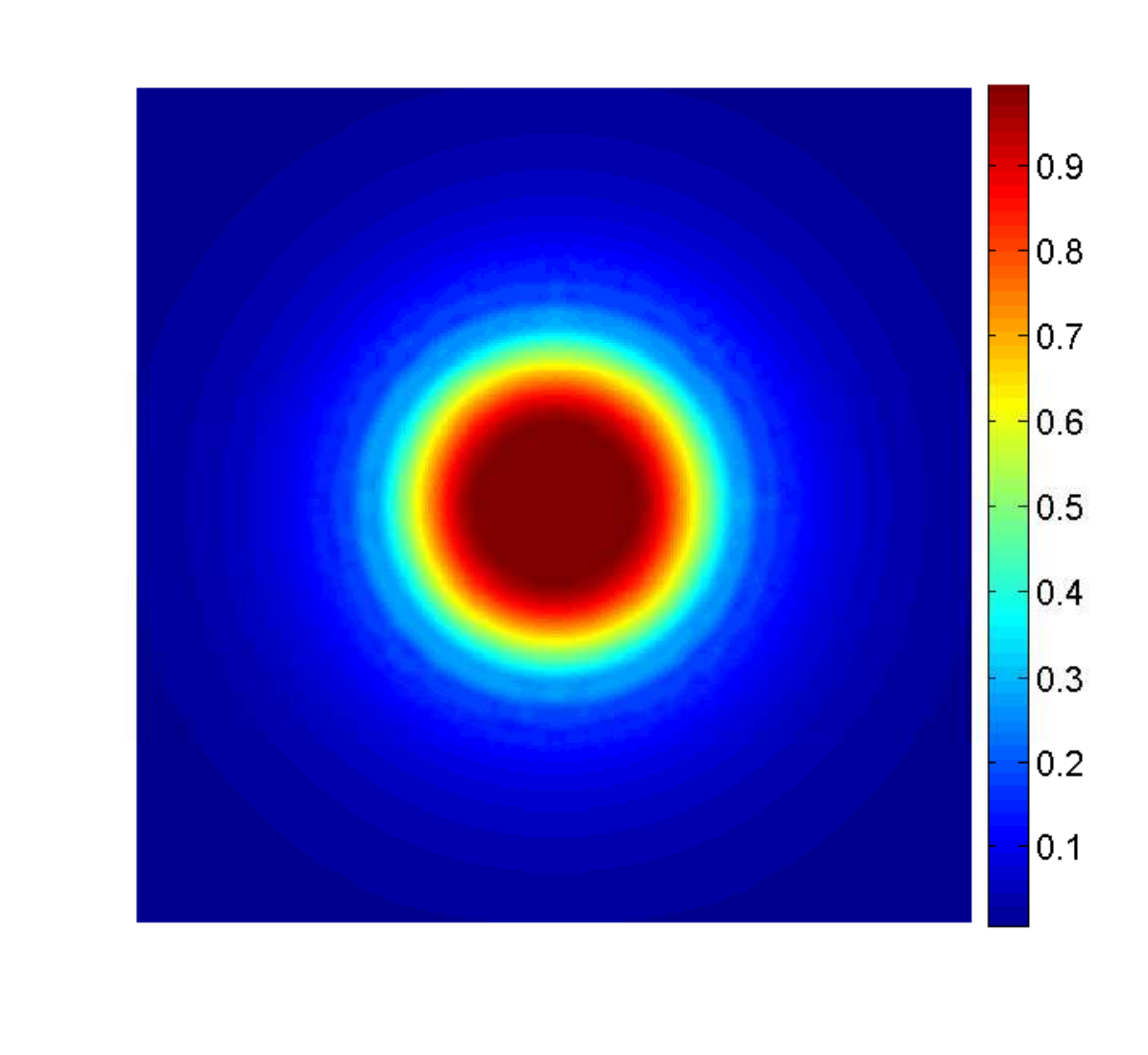}\\
\includegraphics[width=0.35\textwidth]{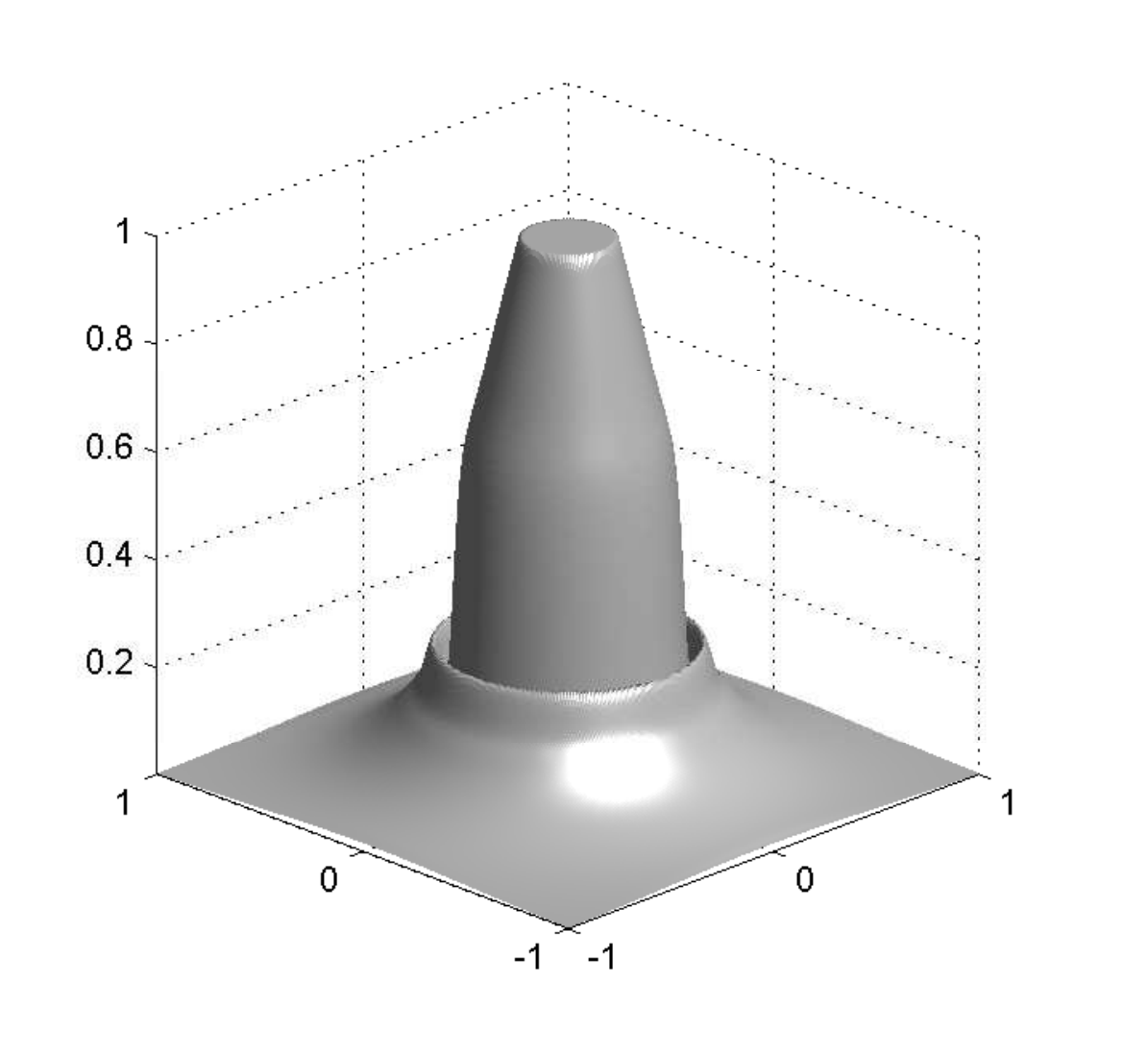}&
\includegraphics[width=0.35\textwidth]{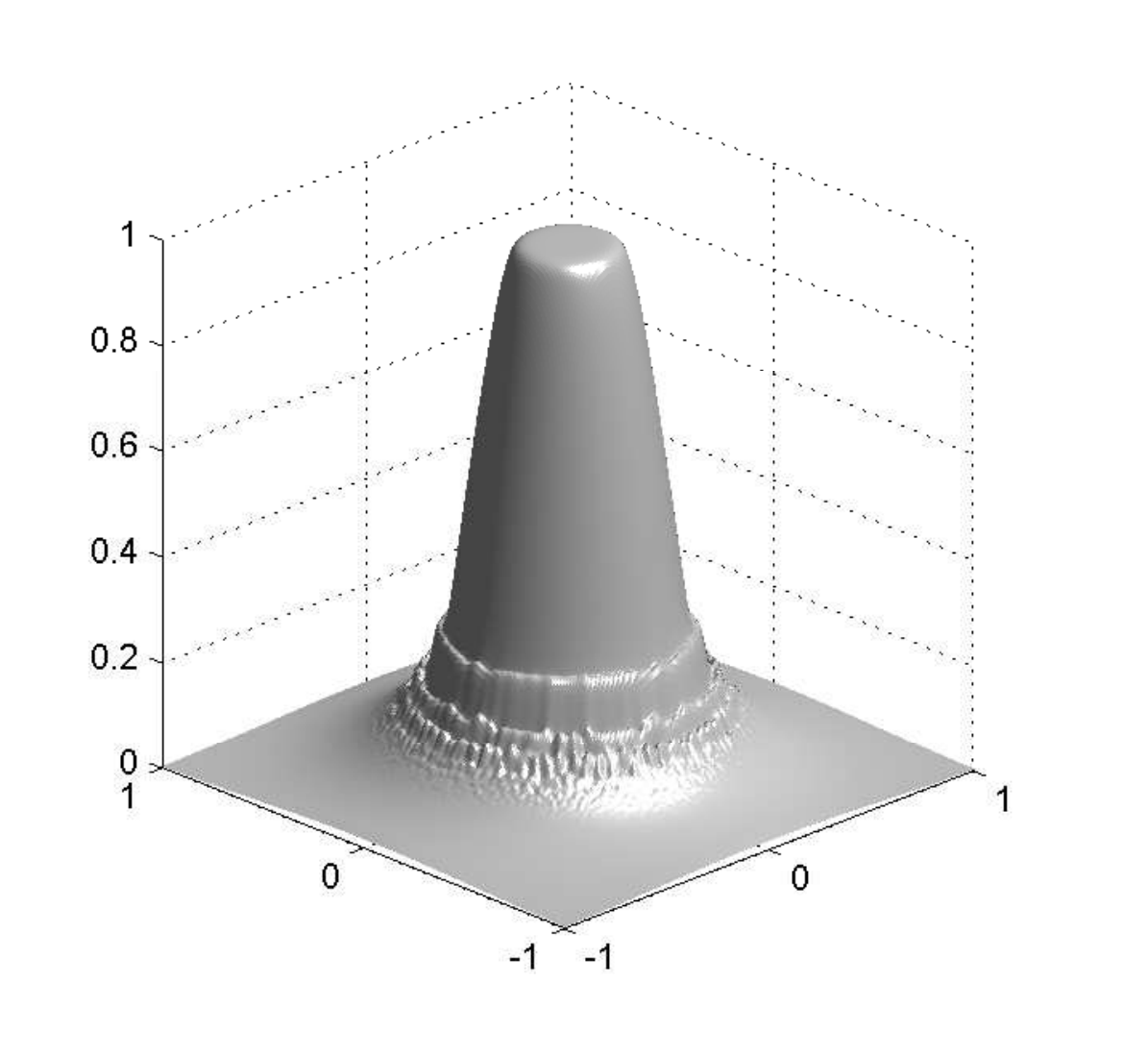}
\end{tabular}
\end{center}
\vspace{-.35cm}
\caption{Example~1: Numerical solution for species $u$, at $t=1.0$
 for $p=2$ (left), and $p=6$ (right).} \label{bbru_fig:ex1}
\end{figure}

In this section, we provide two numerical examples to illustrate how
the approximate solutions of the chemotaxis
model~\eqref{bbru_S1-S2-S3} vary when changing the parameter $p$
from standard nonlinear diffusion ($p=2$) to doubly nonlinear
diffusion ($p>2$). For the discretization of both examples, a
standard first order finite volume method (see the Appendix for
details on the numerical scheme) on a regular mesh of 262144 control
volumes is used. We choose a simple square domain  $\Om=[-1,1]^2$
and use the functions  $a(u)=\epsilon u(1-u)$, $f(u)=(1-u)^2$ and
$g(u,v)=\alpha u-\beta v$, along with parameters that are indicated
separately for each case.

\subsection{Example~1} For the first example, we choose
$\epsilon=0.01$, $\alpha=40$, $\beta=160$, $\chi=0.2$ and $d=0.05$.
The initial condition for the species density is given by
$$u_0(x)=\begin{cases}
1& \text{for $\|x\|\leq 0.2$,}\\
0& \text{otherwise,} \end{cases}$$ 
and the  
chemoattractant  is assumed to have the uniform concentration  
$v_0(x)=4.5$. In a first simulation, we consider the simple case of
$p=2$ and we compare the result with an analogous experiment with
$p=6$. We evolve the system until $t=1.0$, and show in
Figure~\ref{bbru_fig:ex1} a snapshot of the cell density at this
instant for both cases.
\subsection{Example~2}

\begin{figure}[t]
\begin{center}
\begin{tabular}{cc}
\includegraphics[width=0.35\textwidth]{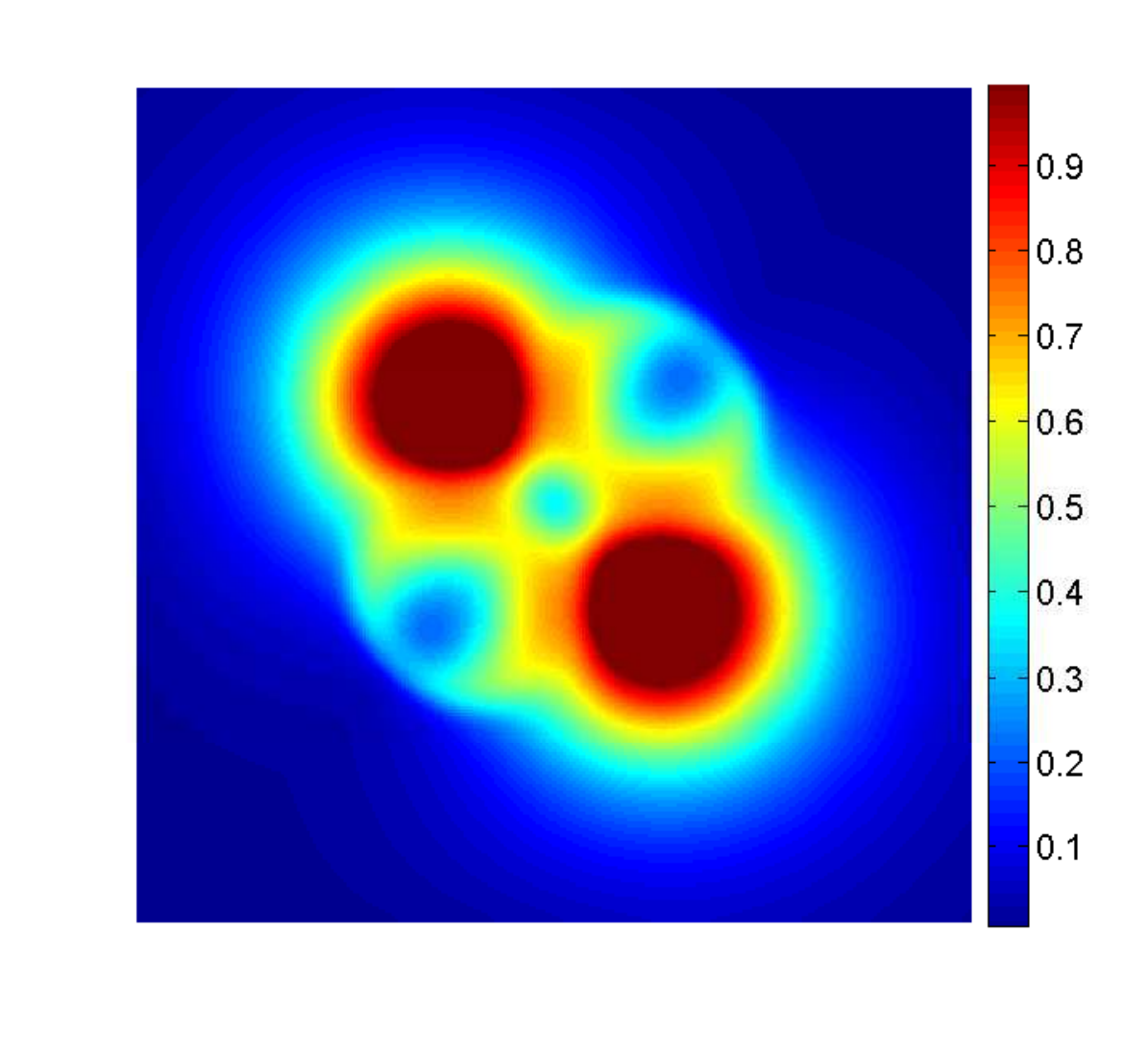}&
\includegraphics[width=0.35\textwidth]{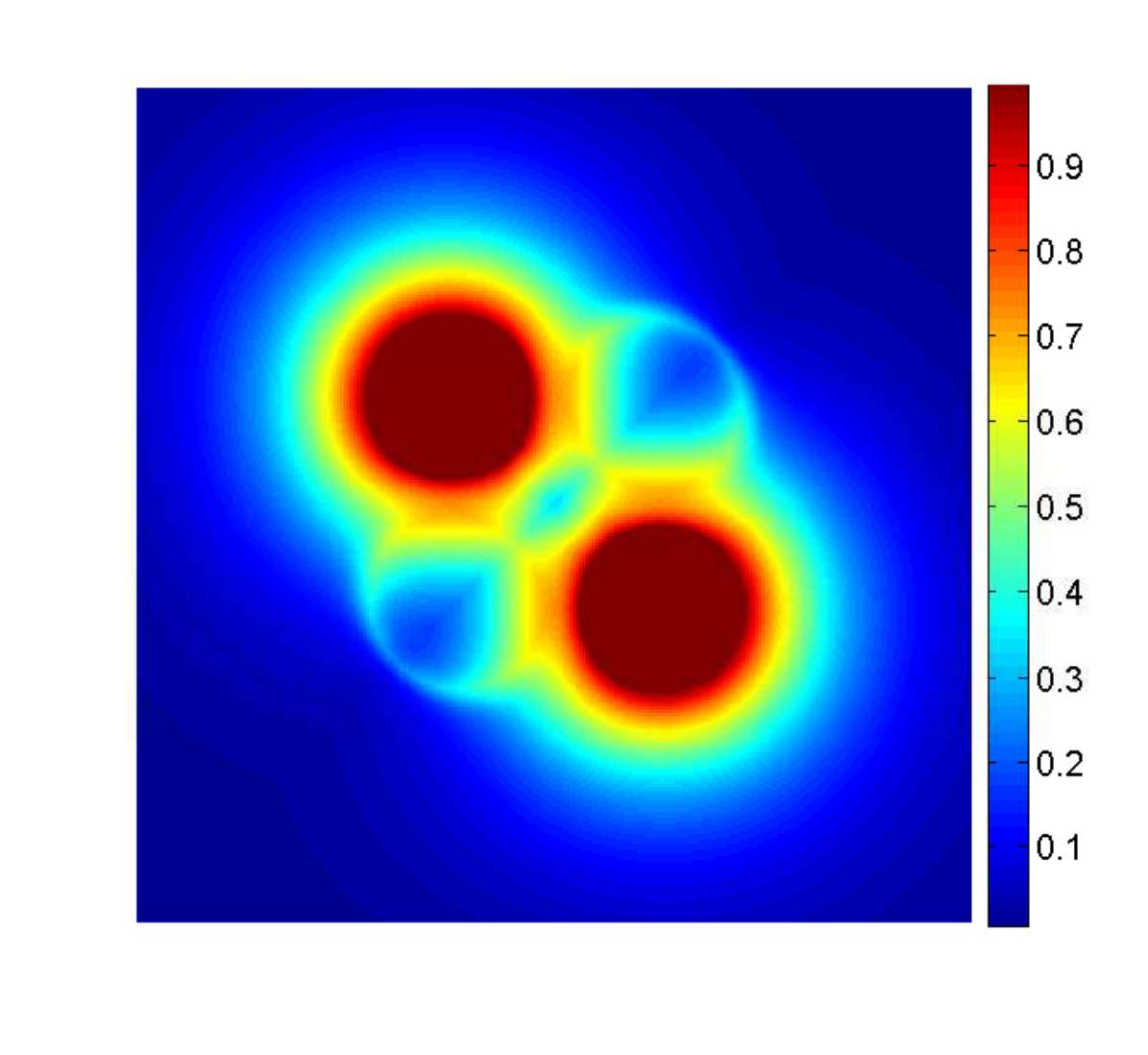}\\
\includegraphics[width=0.35\textwidth]{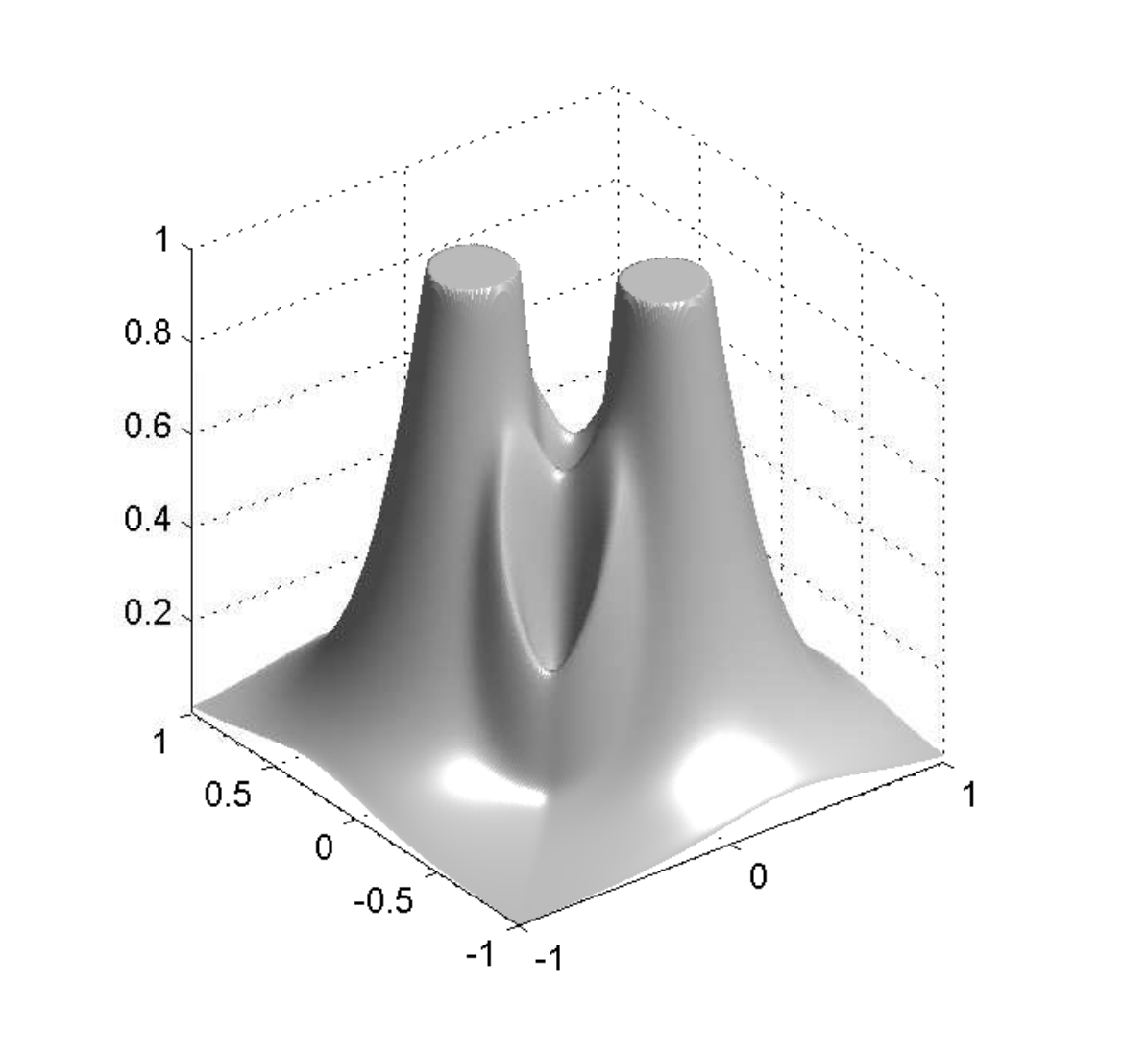}&
\includegraphics[width=0.35\textwidth]{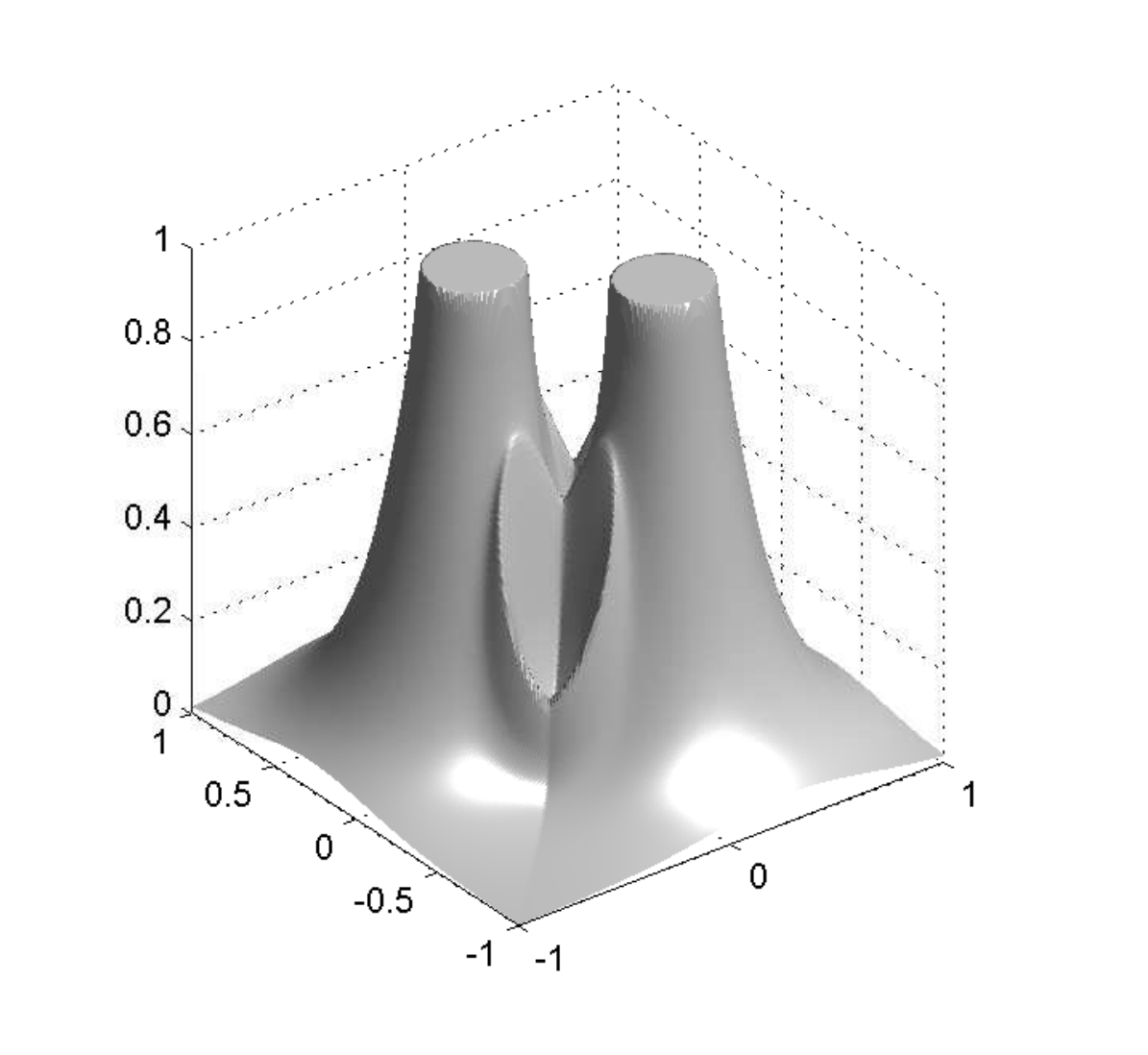}
\end{tabular}
\end{center}
\caption{Example~2: Numerical solution for species $u$, at $t=0.1$
 for $p=2$ (left), and $p=6$ (right).} \label{bbru_fig:ex2aa}
\end{figure}

\begin{figure}[t]
\begin{center}
\begin{tabular}{cc}
\includegraphics[width=0.35\textwidth]{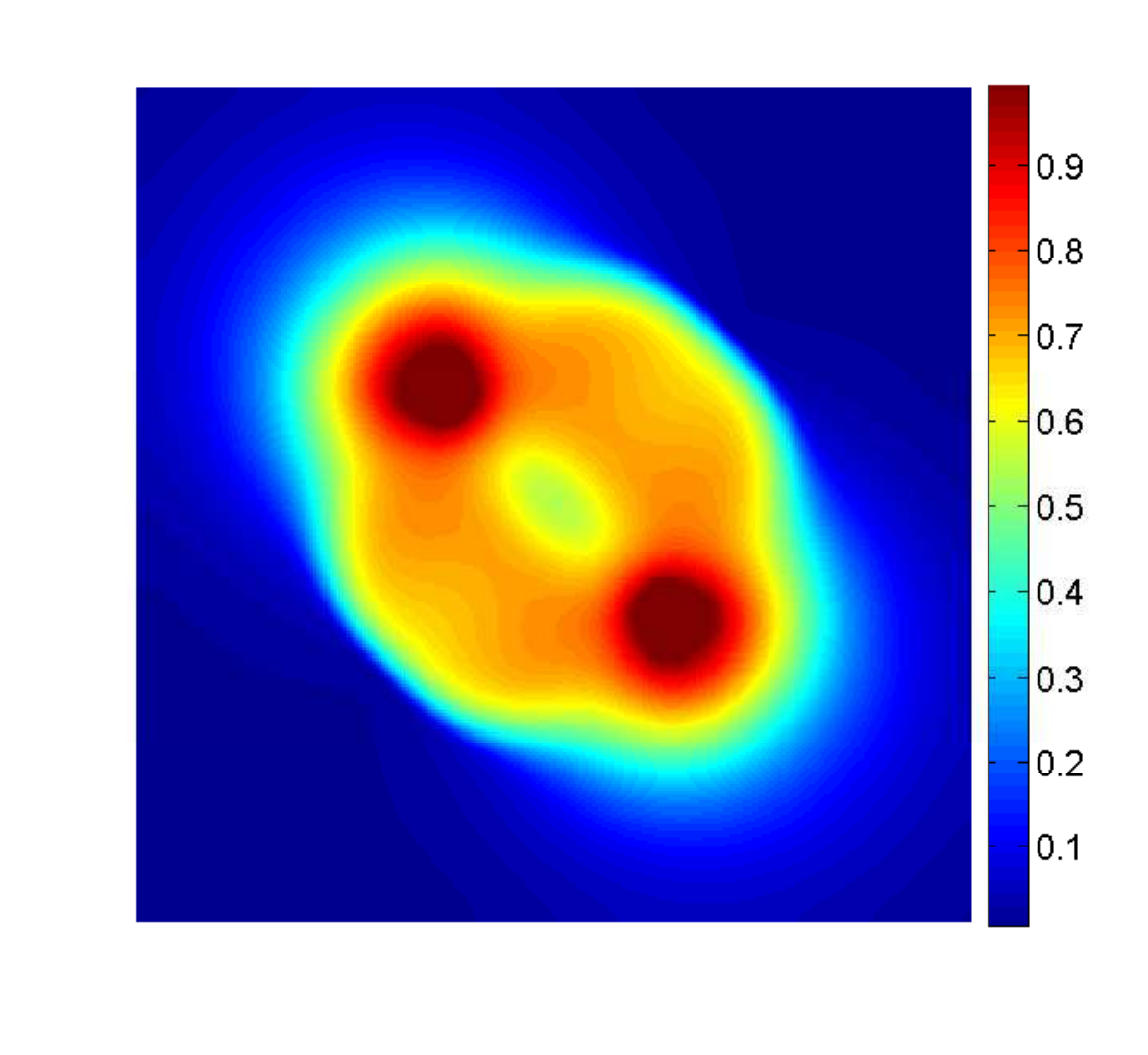}&
\includegraphics[width=0.35\textwidth]{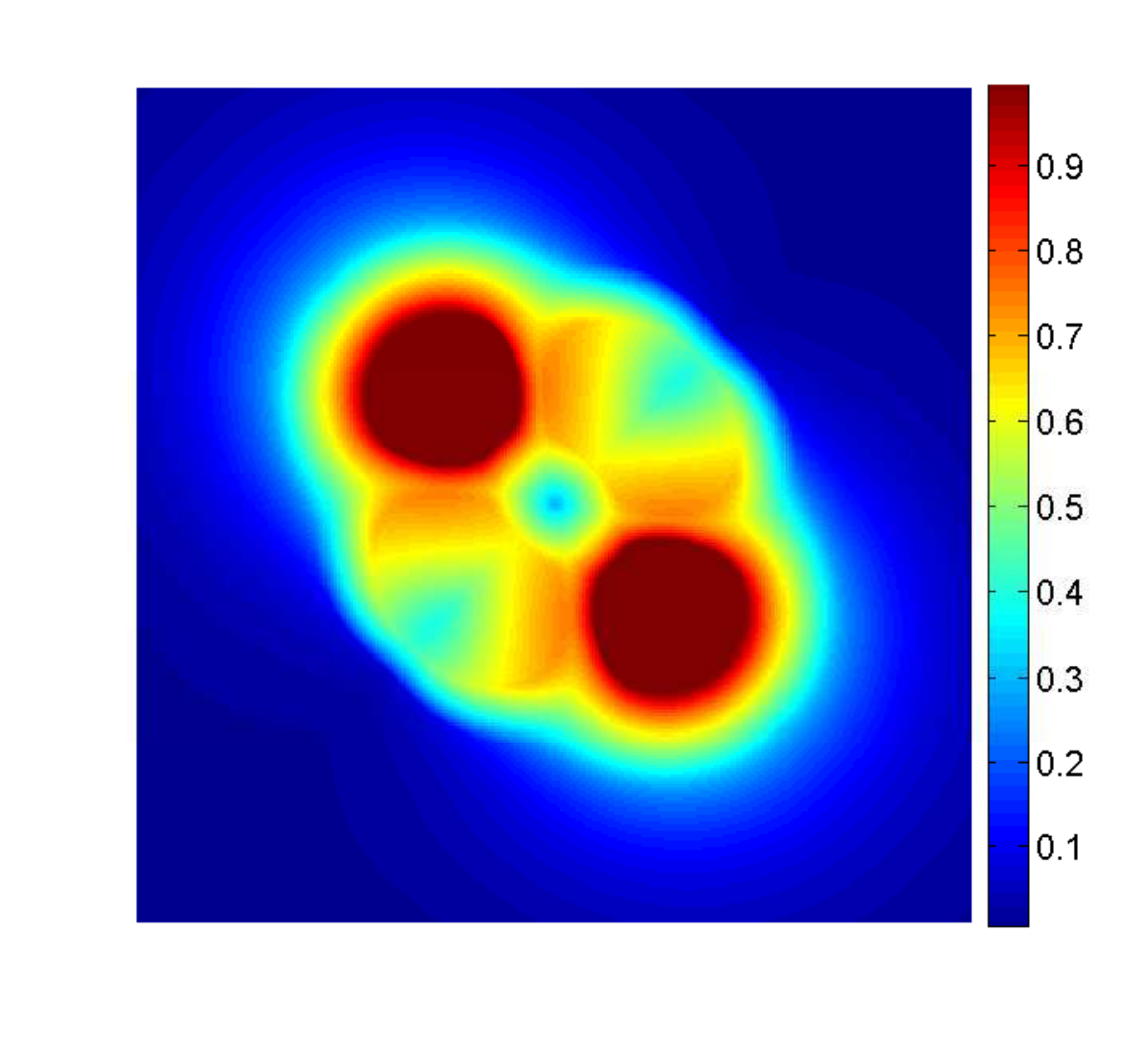}\\
\includegraphics[width=0.35\textwidth]{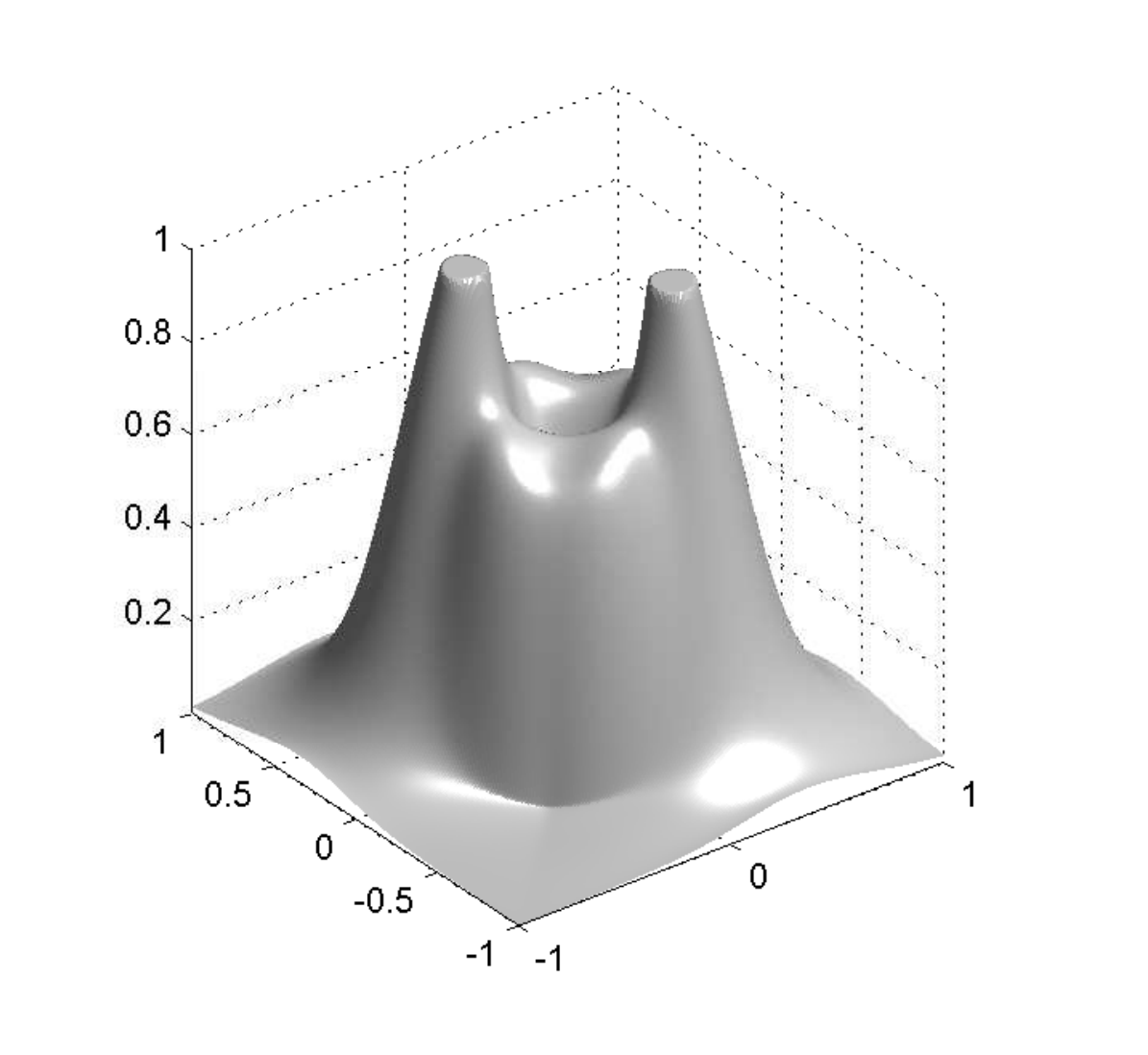}&
\includegraphics[width=0.35\textwidth]{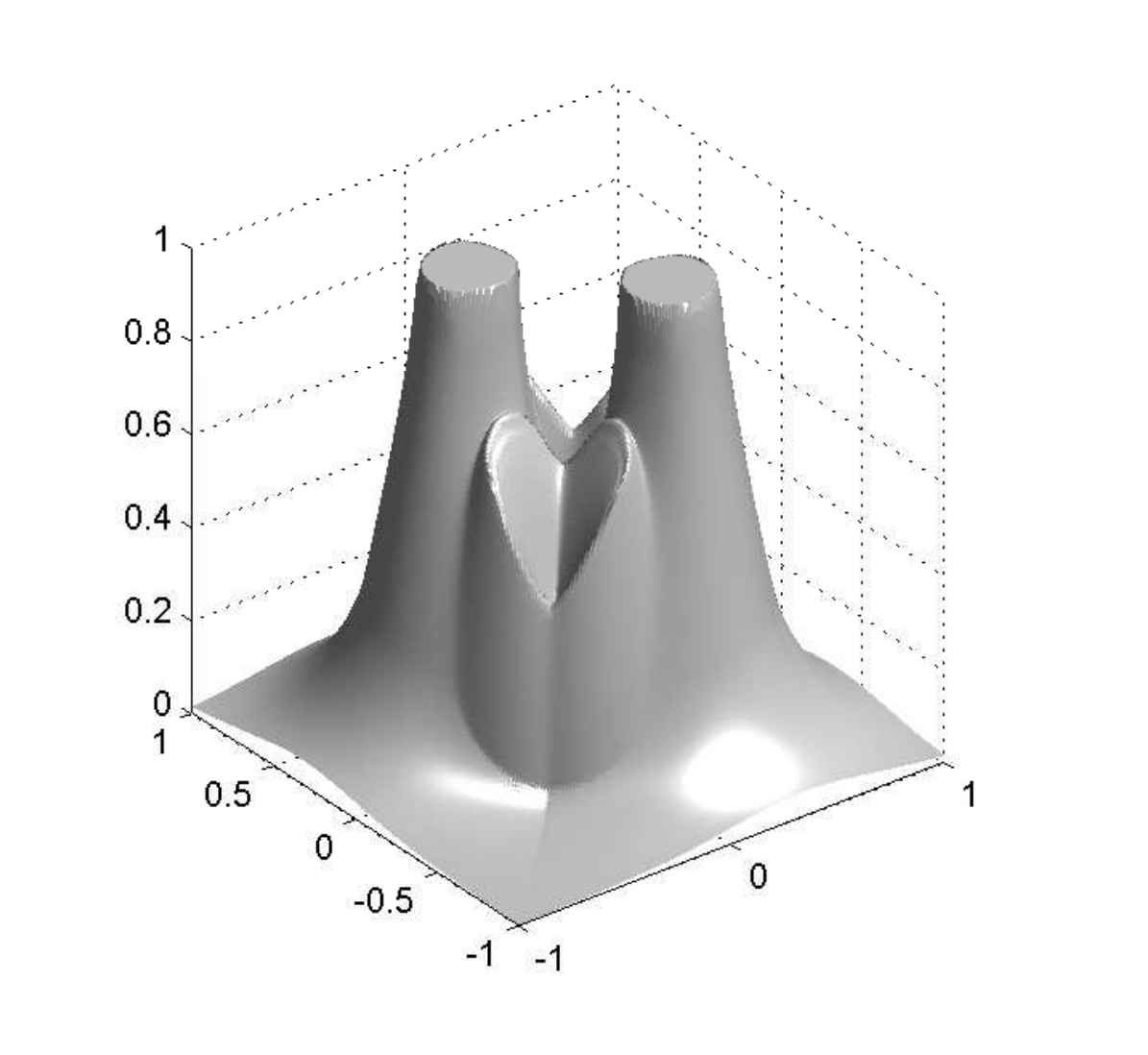}
\end{tabular}
\end{center}
\caption{Example~2: Numerical solution for species $u$, at $t=0.5$
 for $p=2$ (left), and $p=6$ (right).} \label{bbru_fig:ex2ab}
\end{figure}
We now choose
the parameters $\epsilon=0.5$, $\alpha=5$, $\beta=0.5$, $\chi=1$ and
$d=0.25$. The initial condition for the species density is given by
$$u_0(x)=\begin{cases}
1& \text{for $\|x-(-0.25,0.25)\|\leq 0.2$ or $\|x-(0.25,-0.25)\|\leq 0.2$}\\
0& \text{otherwise},\end{cases}
$$
and for the chemoattractant
$$v_0(x)=\begin{cases}
4.5& \text{for $\|x-(0.25,0.25)\|\leq 0.2$ or $\|x+(0.25,0.25)\|\leq 0.2$}\\
0& \text{otherwise.}\end{cases}
$$
The behavior of the system  for the cases
$p=2$ and $p=6$ at different times  is presented  
in  Figures~\ref{bbru_fig:ex2aa}, \ref{bbru_fig:ex2ab}
and~\ref{bbru_fig:ex2b}.

\begin{figure}[t]
\begin{center}
\begin{tabular}{cc}
\includegraphics[width=0.35\textwidth]{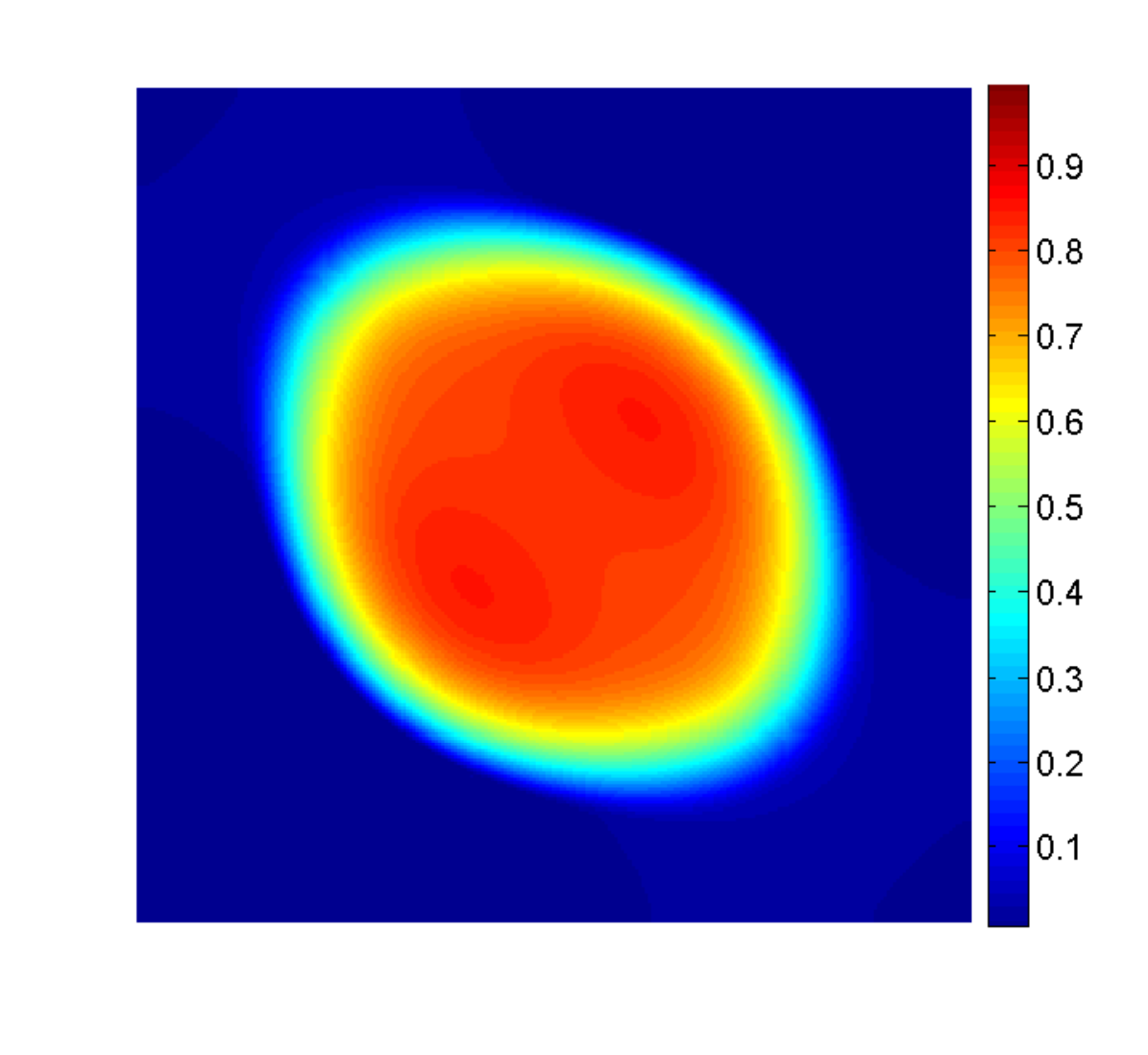}&
\includegraphics[width=0.35\textwidth]{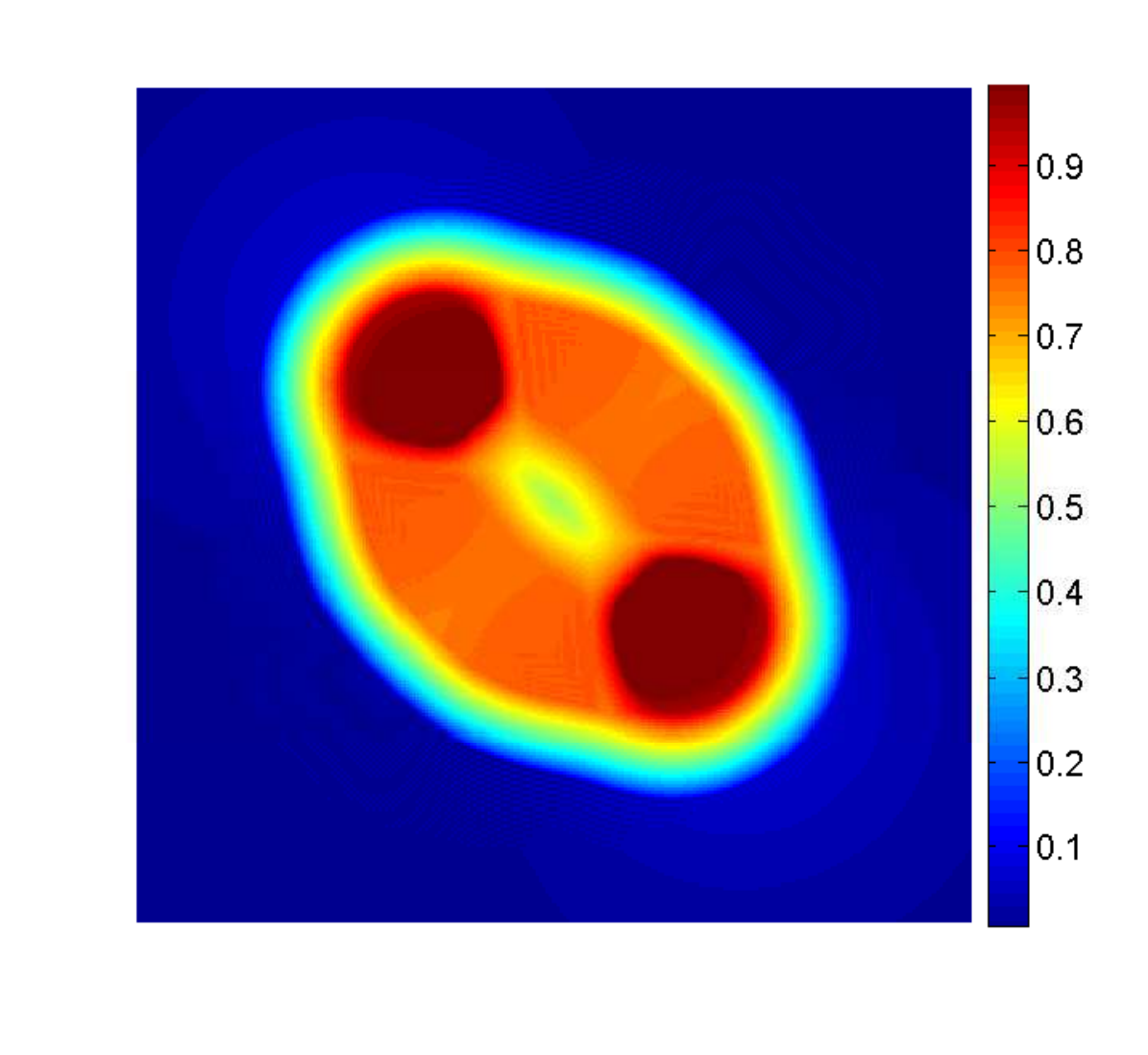}\\
\includegraphics[width=0.35\textwidth]{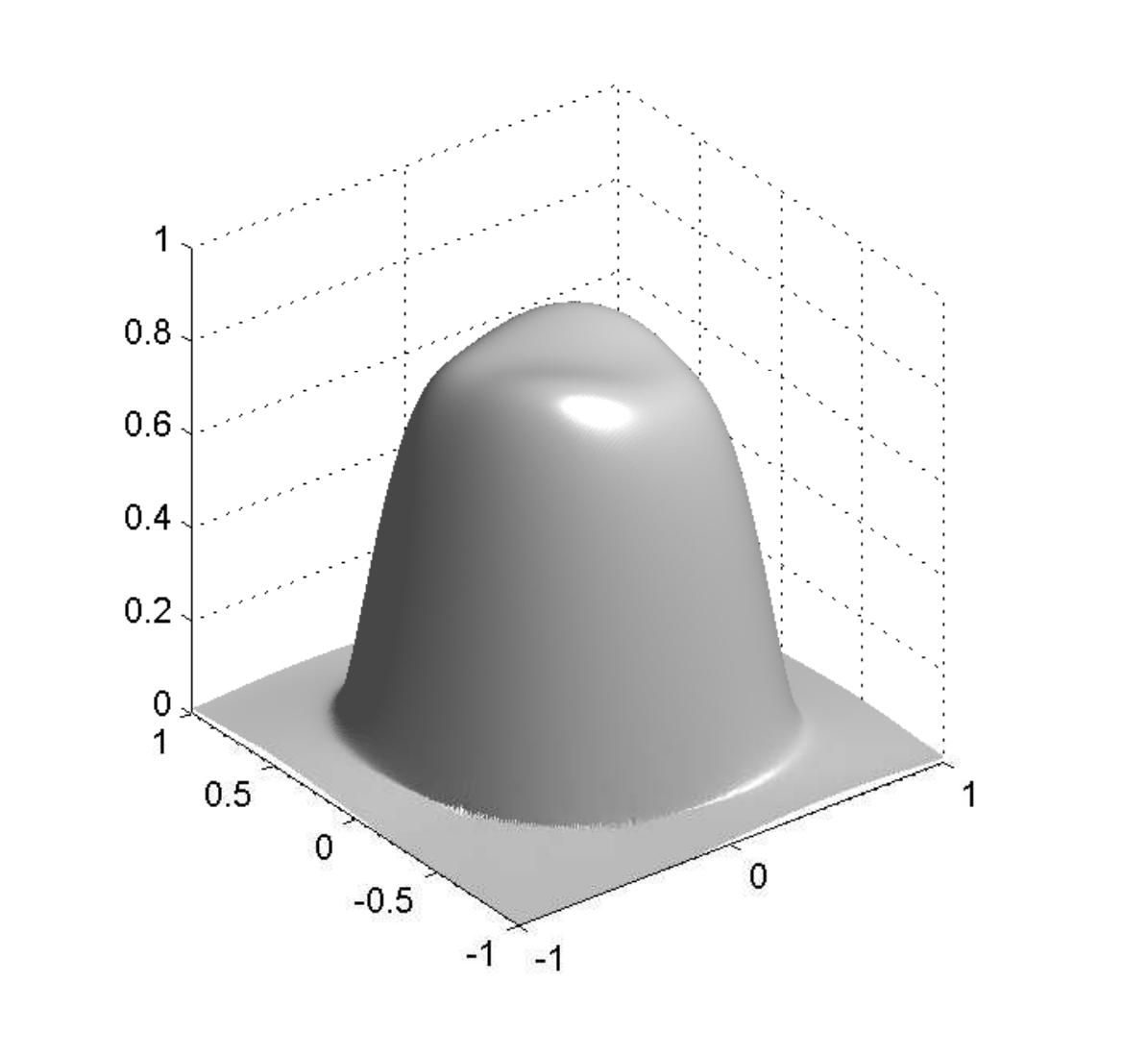}&
\includegraphics[width=0.35\textwidth]{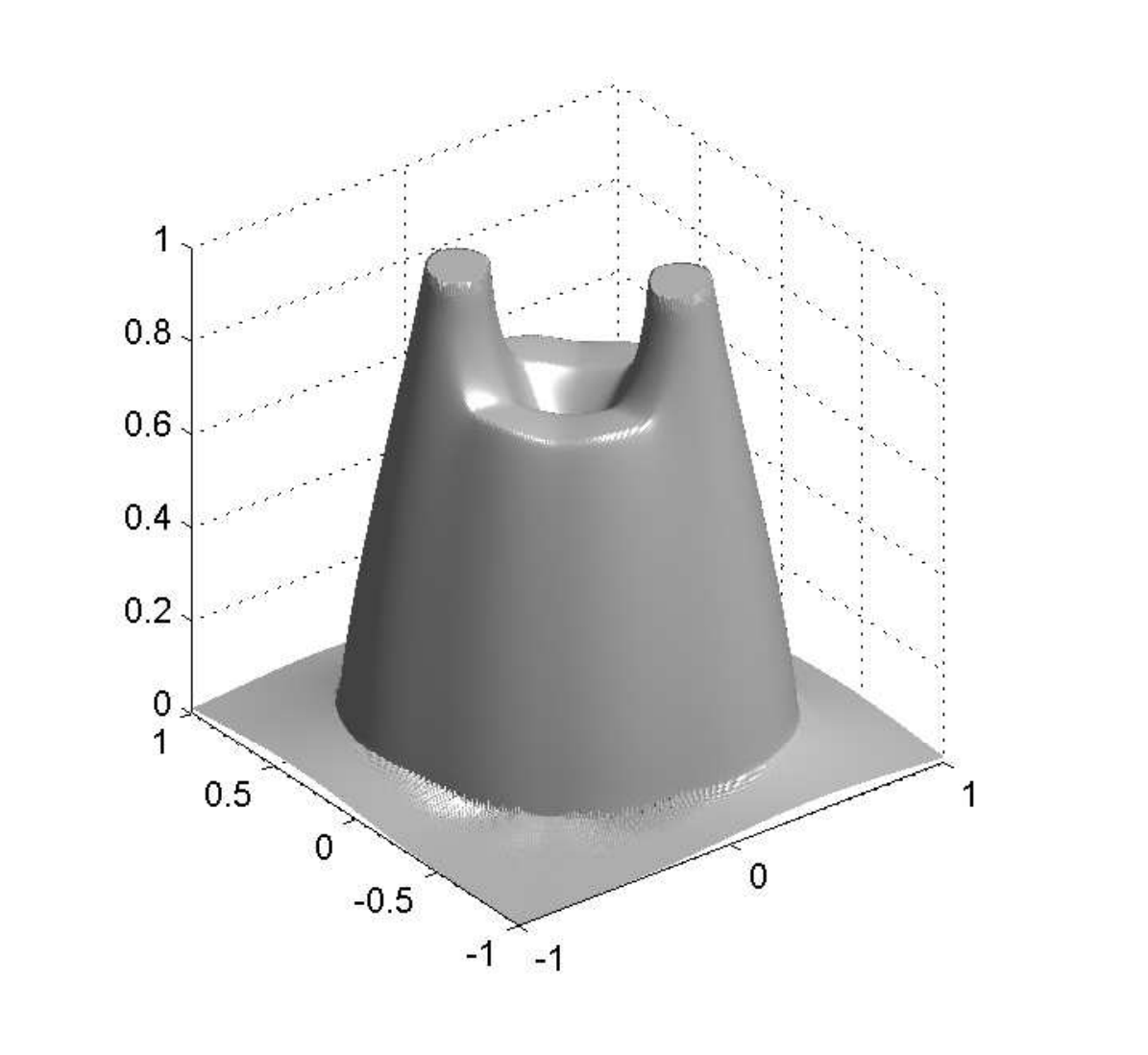}
\end{tabular}
\end{center}
\vspace{-.35cm}
\caption{Example~2: Numerical solution for species $u$, at $t=2.5$
 for $p=2$ (left), and $p=6$ (right).} \label{bbru_fig:ex2b}
\end{figure}

\subsection{Concluding remarks} 
W first mention that, from the previous
examples, one observes that even though the numerical solutions
obtained with $p=2$ differ from those obtained with $p>2$, the
qualitative structure of the solutions remains unchanged. We also
stress that the numerical examples illustrate the effectiveness of
the mechanism of prevention of overcrowding, or volume filling
effect, since all solutions assume values between zero and one only.
In particular, all examples exhibit plateau-like structures where
$u=u_{\mathrm{m}} =1$, at least for small times, which diffuse very
slowly, illustrating that the diffusion coefficient vanishes at
$u=1$ (recall the special form of the functions $a(u)$ and $f(u)$:
they include the factor $(1-u)$, and therefore the species diffusion
and chemotactical cross diffusion terms vanish at $u=0$ and $u=
u_{\mathrm{m}}=1$).

In Example~2, the solution for $p=2$ has a smoother shape than the
one for $p=6$, which exhibits sharp edges. These sharp edges do not
only appear for $u=0$ and $u=u_{\mathrm{m}}$, where one expects
them, due to the degeneracy of the diffusion term and the choice of
initial data, but also for intermediate solution values, as is
illustrated by the plots for $p=6$ of Figures~\ref{bbru_fig:ex2aa}
and \ref{bbru_fig:ex2ab}.

\section*{Acknowledgment}
M. Bendahmane is supported by Fondecyt Project 1070682,
and R. B\"{u}rger  is supported  by Fondecyt Project 1050728
 and Fondap in Applied Mathematics (project 15000001). 
 MB and RB also acknowledge support by CONICYT/INRIA project 
 Bendahmane-Perthame. 
R. Ruiz
acknowledges support by MECESUP project UCO0406 and CMUC. The research of J. Urbano was
supported by CMUC/FCT and Project POCI/MAT/57546/2004. This work was developed
 during a visit of R. Ruiz to the Center for Mathematics at the University
of Coimbra, Portugal.
\section*{Appendix}

The definition of the finite volume method is based on the framework of
\cite{Ey-Gal-Her:book}. An admissible mesh for $\Om$ is given
by a family $\TT$ of control volumes of maximum diameter $h$, a family
of edges $\EE$ and a family of points  $(x_{K})_{K\in \TT}$. For
$K \in \TT$, $x_K$ is the center of~$K$, $\EE_{\text{int}}(K)$ is
the set of edges $\sigma$ of $K$ in the interior of $\TT$,  and
$\EE_{\text{ext}}(K)$ the set of edges of $K$ on the boundary $\partial \Om$.
For all $\sigma \in \EE$, the transmissibility coefficient is
$$
\tau_\sigma=\begin{cases} \displaystyle
\frac{|\sigma|}{d(x_K,x_L)} & \text{for $\sigma \in
  \EE_{\text{int}}(K)$,
 $\sigma=K|L$,} \\[2mm]  \displaystyle
\frac{|\sigma|}{d(x_K,\sigma)}& \text{for $\sigma \in
\EE_{\text{ext}}(K)$,}
\end{cases}$$
where $K|L$ denotes the common edge of neighboring
  finite volumes~$K$ and~$L$.
For $K \in \TT$ and $\sigma=K|L \in \EE(K)$ with common vertexes
$(a_{\ell,K,L})_{1\le \ell\le I}$ with $I \in \N \backslash \{0\}$,
let $T_\sigma$ ($T^{\text{ext}}_{K,\sigma}$ for
$\sigma\in \EE_{\text{ext}}(K)$, respectively) be the open and convex polygon
built by the convex envelope with vertices $(x_K,x_L)$ ($x_K$,  respectively) and
$(a_{\ell,K,L})_{1\le \ell\le I}$. The domain $\Om$ can be decomposed into
\begin{align*}
\overline{\Om}=\cup_{K\in \TT}
\bigl((\cup_{L\in N(K)}\overline{T}_{K,L})\cup
 (\cup_{\sigma \in
\EE_{\text{ext}}(K)}\overline{T}^{\text{ext}}_{K,\sigma})\bigr).
\end{align*}
For all $K \in \TT$, the approximation $\Grad_h u_{K,\sigma}$ of $\Grad u$
is defined by
$$\Grad_h u^n_{K,\sigma}:=\begin{cases}
u^n_{L}-u^n_{K} & \text{if $\sigma=K|L\in\EE_{\text{int}}(K)$}, \\
    0   & \text{if $\sigma \in\EE_{\text{ext}}(K)$}.
\end{cases}$$
To discretize \eqref{bbru_S1-S2-S3}, we choose an admissible
 mesh of $\Om$ and  a time step size
$\Delta t>0$. If $M_T>0$ is the smallest integer such that $M_T\Delta
t\ge T$, then $t^n:=n \Delta t$ for $n\in  \{0,\ldots,M_T\}$.

We define cell averages of the unknowns $A(u)$, $f(u)$ and $g(u,v)$ over $K\in \TT$ :
\begin{align*}
&A_{K}^{n+1}:=\frac{1}{\Delta t |K|}\int_{t^n}^{t^{n+1}}
\int_KA \bigl(u(x,t)\bigr)\dx\dt, \quad g_{K}^{n+1}:=\frac{1}{\Delta t |K|}\int_{t^n}^{t^{n+1}}
\int_Kg \bigl(u(x,t),v(x,t) \bigr)\dx\dt, \\
&\qquad \qquad \qquad \qquad f_K^{n+1}:=\frac{1}{\Delta t |K|}\int_{t^n}^{t^{n+1}}
\int_Kf\bigl(u(x,t)\bigr)\dx\dt,
\end{align*}
and the initial conditions are discretized by
\begin{align*}
u_K^0=\frac{1}{|K|} \int_{K} u_0(x) \dx,\quad
v_K^0=\frac{1}{|K|} \int_{K} v_0(x) \dx.
\end{align*}
We now give the finite  volume scheme employed to advance the numerical
solution from $t^n$ to $t^{n+1}$, which is based on a simple explicit
Euler time discretization. Assuming that  at $t=t^n$, the pairs
$(u_{K}^{n},v_{K}^{n})$ are known for all $K \in \TT$,  we compute
$(u_{K}^{n+1},v_{K}^{n+1})$  from
\begin{align*}
|K|\frac{u^{n+1}_K-u^{n}_K}{\Delta t}&=\sum_{\sigma\in \EE(K)}\tau_\sigma
\left|\Grad_h A^{n}_{K,\sigma}\right|_h^{p-2}\Grad_h A^{n}_{K,\sigma}
+\chi \sum_{\sigma\in \EE(K)}\tau_\sigma\left[\left(\Grad_h v^{n}_{K,\sigma}\right)^+
u^{n}_{K}f^n_K-\left(\Grad_h v^{n}_{K,\sigma}\right)^-u^{n}_{L}f^n_L\right],\\
|K|\frac{v^{n+1}_K-v^{n}_K}{\Delta t}&=\sum_{\sigma\in \EE(K)}\tau_\sigma
\Grad_h v^{n}_{K,\sigma}+|K|g^n_K.
\end{align*}
Here $|\cdot|_h$ denotes the discrete Euclidean norm. The Neumann boundary conditions
are taken into account by imposing zero fluxes on the external edges.


\end{document}